\documentclass[12pt]{amsart}
\usepackage[utf8]{inputenc}

\usepackage{amsmath,amsthm,amssymb,mathscinet,bbm}
\usepackage{parskip}
\usepackage{geometry}
\usepackage{xcolor} 
\usepackage{xfrac, nicefrac}
\usepackage{exscale, relsize}

\usepackage{mathtools} 
\usepackage[hidelinks]{hyperref}
\author{Akash Singha Roy}
\address{Department of Mathematics \\ University of Georgia \\ Athens, GA 30602}
\email{akash01s.roy@gmail.com}

\subjclass[2020]{Primary 11A25; Secondary 11N36, 11N37, 11N64, 11N69}

\renewcommand\phi\varphi
\usepackage{accents}

\renewcommand{\pod}[1]{\allowbreak\mathchoice
  {\if@display \mkern 18mu\else \mkern 8mu\fi (#1)}
  {\if@display \mkern 18mu\else \mkern 8mu\fi (#1)}
  {\mkern4mu(#1)}
  {\mkern4mu(#1)}
}
\usepackage{graphicx}
\DeclareMathAlphabet{\curly}{U}{rsfs}{m}{n}

\newcommand{\F}{\mathbb{F}}
\newcommand\Z{\mathbb{Z}}
\allowdisplaybreaks

\newcommand\NatNos{\mathbb N}

\renewcommand\v{\mathbf{v}}

\newcommand\Q{\mathbb{Q}}

\newtheorem{thm}{Theorem}[section]
\newtheorem{cor}[thm]{Corollary}
\newtheorem{prop}[thm]{Proposition}
\newtheorem{lem}[thm]{Lemma}

\theoremstyle{remark}

\newcommand\ord{\mathrm{ord}}

\newcommand\Fell{\F_\ell}
\newcommand\FellT{\F_\ell[T]}

\newcommand\logxK{(\log x)^K}
\newcommand\Qg{\mathcal Q_g}
\newcommand\Qgfam{\mathcal Q_{(g_1, \dots , g_M)}}

\newcommand\Dmin{D_{\text{min}}}
\newcommand\CGsub{C_{\widehat G}}
\newcommand\sm\setminus
\newcommand\COG{C_0(\widehat G)}
\newcommand\CG{C(\widehat{G})}

\newcommand\VqmPr{\mathcal V_{q, m}'}
\newcommand\VQZmPr{\mathcal V_{Q_0, m}'}
\newcommand\VJMqbigim{\mathcal V_{J, M}\left(q; (b_i-g_i(m))_{i=1}^M\right)}
\newcommand\VJMqwi{\mathcal V_{J, M}\left(q; (w_i)_{i=1}^M\right)}
\newcommand\VNMqwi{\mathcal V_{N, M}\left(q; (w_i)_{i=1}^M\right)}
\newcommand\VNMQZwi{\mathcal V_{N, M}\left(Q_0; (w_i)_{i=1}^M\right)}
\newcommand\VJMQZwi{\mathcal V_{J, M}\left(Q_0; (w_i)_{i=1}^M\right)}
\newcommand\largesum{\mathlarger{\sum}}
\newcommand\largeprod{\mathlarger{\prod}}
\newcommand\qtil{\widetilde q}
\newcommand\chiZqtil{\chi_{0, \qtil}}
\newcommand\xiG{\xi_{\widehat{G}}}
\newcommand\MDOne{MD+1}
\newcommand\OmStqn{\Omega^*_{>q}(n)}
\newcommand\OmStEN{\Omega^*_E(N)}
\newcommand\GiPrSet{\{G_i'\}_{1 \le i \le M}}
\newcommand\critsetAFell{\mathcal A_{F, \ell}}

\newcommand\chiZell{\chi_{0, \ell}}
\newcommand\tellF{t_\ell(F)}
\newcommand\ordell{\ord_\ell}
\newcommand\ordellFPr{\ord_\ell(F')}
\newcommand\Ftil{\widetilde{F}}
\newcommand\FtilT{\Ftil(T)}
\newcommand\Ftilv{\Ftil(v)}

\numberwithin{equation}{section}
\begin{document}
\title[Equidistribution of families of polynomially-defined additive functions]{Joint distribution in residue classes of families of polynomially-defined additive functions}
\begin{abstract} 
Let $g_1, \dots , g_M$ be additive functions for which there exist nonconstant polynomials $G_1, \dots , G_M$ satisfying $g_i(p) = G_i(p)$ for all primes $p$ and all $i \in \{1, \dots , M\}$. Under fairly general and nearly optimal hypotheses, we show that the functions $g_1, \dots , g_M$ are jointly
equidistributed among the residue classes to moduli $q$ varying uniformly up to a fixed but arbitrary power of $\log x$. Thus, we obtain analogues of the Siegel-Walfisz Theorem for primes in arithmetic progressions, but with primes replaced by values of such additive functions. 
Our results partially extend work of Delange from fixed moduli to varying moduli, and also generalize recent work done for a single additive function. 
\end{abstract}
\keywords{additive function, uniform distribution, equidistribution, joint distribution, joint equidistribution}
\maketitle
\section{Introduction}  
We say that an integer-valued arithmetic function $g$ is \textsf{uniformly distributed} (or \textsf{equidistributed}) modulo $q$ if 
\begin{equation}\label{eq:equiddef}
\#\{n \le x: g(n)\equiv b\pmod q\} \sim \frac xq \quad\text{as $x\to\infty$},
\end{equation}
for each residue class $b\bmod q$. As a nontrivial example, it is a result due to Pillai \cite{pillai40} that the function $\Omega(n):=\sum_{p^k\parallel n} k$ counting the prime factors of $n$ with multiplicity is uniformly distributed modulo any positive integer $q$. For general additive functions, a satisfactory characterization was obtained by Delange \cite{delange69} in 1969 for when an additive function $g$ is uniformly distributed to a fixed integer modulus $q$: his criterion involved the sums $\sum_{p:~d \nmid f(p)} 1/p$ for divisors $d>1$ of $q$ (we state the result precisely in the next section). This result shows, for instance, that the function $A(n) \coloneqq \sum_{p^k \parallel n} kp$ (the sum of the prime divisors of $n$ counted with multiplicity) is equidistributed among the residue classes of any fixed integer modulus. 

We say that a family $g_1, \dots , g_M$ of integer-valued arithmetic functions is \textsf{jointly equidistributed} modulo $q$ if 
\[ \#\{n \le x: \forall i \in [M], ~ ~ g_i(n)\equiv b_i\pmod q\} \sim \frac{x}{q^M} \quad\text{as $x\to\infty$},\]
for all residues $b_1, \dots , b_M \bmod q$. (Here $[M]$ denotes the set $\{1, \cdots, M\}$.) One can similarly ask whether it is possible to characterize families of additive functions $g_1, \dots , g_M$ that are jointly equidistributed to a fixed integer modulus $q$. Such a characterization was achieved by Delange in \cite{delange74} where he showed that the joint equidistribution of $g_1, \dots , g_M$ modulo $q$ is equivalent to the equidistribution of certain integral linear combinations of $g_1, \dots , g_M$ mod $q$; see Proposition \ref{prop:DelangeJoint} for the precise statement.

In all of the aforementioned results, the modulus $q$ is assumed fixed. A natural question is what happens when the modulus $q$ is allowed to vary; in particular, whether equidistribution continues to hold as $q$ varies uniformly in a suitable range depending on the stopping point of inputs (what we have been calling ``$x$"). A reasonable goal in such an investigation would be to seek analogues of the Siegel--Walfisz Theorem for primes in arithmetic progressions, according to which the primes up to $x$ are asymptotically equidistributed among the coprime residue classes modulo $q$, uniformly for $q$ varying up to any fixed power of $\log x$. In other words, it is reasonable to look for a version of the Siegel--Walfisz theorem, but with primes replaced by values of additive functions. 

In order to make things precise, we will say that given $K \ge 1$, an integer-valued arithmetic function $g$ is \textsf{equidistributed mod $q$ uniformly for $q \le (\log x)^K$} if the relation \eqref{eq:equiddef} holds uniformly in moduli $q \le (\log x)^K$ and in residue classes $b$ mod $q$. Explicitly, this means that for any $\epsilon>0$, there exists $X(\epsilon)>0$ such that the ratio of the left hand side of \eqref{eq:equiddef} to the right hand side lies in $(1-\epsilon, 1+\epsilon)$ for all $x>X(\epsilon)$, all $q \le (\log x)^K$ and all residue classes $b$ mod $q$. This definition extends naturally to families of arithmetic functions, and we analogously define what it means for a given family $g_1, \dots , g_M$ of arithmetic functions to be \textsf{jointly equidistributed mod $q$, uniformly for $q \le (\log x)^K$}. 

Our aim in this paper is to study this phenomenon of joint equidistribution (to uniformly varying moduli) for a large class of additive functions, namely those which can be defined by the values of a polynomial at the primes. We say that an additive function $g \colon \NatNos \rightarrow \Z$ is \textsf{polynomially-defined} if there exists a nonconstant polynomial $G \in \Z[T]$ satisfying $g(p)=G(p)$ for all primes $p$; we will then say that $g$ is \textsf{defined by (the polynomial)} $G$. For example, both the additive functions $\beta(n) \coloneqq \sum_{p \mid n} p$ and $A(n) = \sum_{p^k \parallel n} kp$ are defined by the polynomial $G(T) = T$. 

The equidistribution of a single polynomially-defined additive function with uniformity in modulus seems to have been first studied in \cite{PSR}. In that paper, Hal\'asz's mean value theorem is used to show that for any fixed $\delta>0$, the function $A(n)$ is equidistributed mod $q$ uniformly for $q\le (\log x)^{\frac12-\delta}$. In \cite{PSR23}, this has been improved to $q\le (\log x)^K$ for the function $A(n)$, the full range permitted by the Siegel--Walfisz theorem. The method relies on exploiting an ergodicity (or mixing) phenomenon in the multiplicative group mod $q$, and was primarily used in \cite{PSR23} to study the distribution of polynomially defined multiplicative functions among the coprime residue classes to moduli $q$ varying up to any fixed power of $\log x$. Recent work of Akande \cite{akande23} investigates the distribution of a single general polynomially-defined additive function (see the paragraph following the statement of Theorem \ref{thm:UnrestrictedInput}). To do this, he suitably modifies the method in \cite{PSR23} by means of certain exponential sum estimates.

In the first main result of this paper, we shall generalize the results in \cite{akande23} to families $g_1, \dots , g_M$ of additive functions defined by nonconstant polynomials $G_1, \dots , G_M \in \Z[T]$ respectively, thus extending Delange's work \cite{delange74} to uniformly varying moduli, for families of polynomially-defined additive functions. To this end, let $\Qgfam$ denote the set of moduli $q$ such that $g_1, \dots , g_M$ are jointly equidistributed mod $q$. Under general conditions, we will show that $g_1, \dots , g_M$ are also jointly equidistributed mod $q$ uniformly for $q \le \logxK$ lying in $\Qgfam$. For technical reasons to be elaborated on later (see Theorem \ref{thm:LIHNecess}), we will assume in our main results (Theorems \ref{thm:UnrestrictedInput}, \ref{thm:restrictediput_generalq} and \ref{thm:restrictediput_squarefreeq}) that the derivatives of $G_i$ are linearly independent over $\Q$. This amounts to assuming that no nontrivial $\Z$-linear combination of the $G_i$ reduces to a constant in $\Z[T]$, or in other words, that the polynomials $\{G_i(T)-G_i(0): 1 \le i \le M\} \subset \Q[T]$ are $\Q$-linearly independent. (For $M=1$, this simply amounts to $G_1$ being nonconstant.) In particular, this hypothesis forces the maximum of the degrees of the $G_i$ to be no less than $M$.

Our first main result shows that $g_1, \dots , g_M$ are jointly equidistributed to moduli $q$ lying in $\Qgfam$ varying uniformly up to a small power of $\log x$. In what follows, we denote by $D$ and $\Dmin$ the maximum and the minimum of the degrees of $G_1, \dots , G_M$ respectively,\footnote{The asymmetry in notation is due to the much greater frequency of the appearance of $D$ in our results, as compared to $\Dmin$.} so that by the above discussion, $D \ge M$. 
\begin{thm}\label{thm:UnrestrictedInput}
Fix $K \ge 1$, $\delta \in (0, 1]$ and an integer $M \ge 1$. Let $g_1, \dots , g_M$ be additive functions defined by the polynomials $G_1, \dots , G_M$ such that the polynomials $\{G_i'\}_{1 \le i \le M} \subset \Z[T]$ are $\Q$-linearly independent. Then $g_1, \dots , g_M$ are jointly equidistributed modulo $q$, uniformly for $q \le (\log x)^K$ lying in $\Qgfam$, under any of the following additional conditions.
\begin{enumerate}
    \item[(i)] $M=1$, and either $q$ is squarefree or $G_1$ is linear. 
    \item[(ii)] $M \ge 2$, $q \le (\log x)^{(1-\delta)/(M-1)}$, and either $q$ is squarefree or at least one of $G_1, \dots , G_M$ is linear.
    \item[(iii)] $q \le (\log x)^{(1-\delta)(M-1/\Dmin)^{-1}}$.
\end{enumerate}
\end{thm}
Subpart (i) and the special case $M=1$ of subpart (iii) are the main results in \cite{akande23}, but we have included them here in order to give a self-contained and unified treatment. These assertions will of course be automatically established by our method as well. However, our method is significantly different from \cite{akande23} and there are several additional ideas required to generalize these special cases to our theorem above.

In subsection \ref{subsec:OptimalityUnrestricted}, we shall show that the ranges of $q$ in the subparts of the above theorem are all essentially optimal. In the constructions described there, the obstructions to uniformity will come from the prime inputs $p$. Modifying the construction slightly, we could produce obstructions of the form $mp$ with $m$ fixed or even slowly growing with $x$. Our next two results point out that the inputs $n$ with too few `large' prime factors do indeed present the key obstructions to uniformity. In other words, we show that uniformity in $q$ up to an arbitrary power of $\log x$ can be restored by restricting the set of inputs $n$ to those having sufficiently many prime divisors (counted with multiplicty) exceeding $q$. 

To make this precise, we write $P(n)$ for the largest prime divisor of $n$, with the convention that $P(1)=1$. We set $P_1(n) \coloneqq P(n)$ and define, inductively, $P_k(n) \coloneqq P_{k-1}(n/P(n))$. Thus, $P_k(n)$ is the $k$th largest prime factor of $n$ (counted with multiplicity), with $P_k(n)=1$ if $\Omega(n) < k$. 
\begin{thm}\label{thm:restrictediput_generalq}
Fix $K, M \ge 1$ and let $g_1, \dots , g_M$ be additive functions defined by the polynomials $G_1, \dots , G_M$, such that $\{G_i'\}_{1 \le i \le M} \subset \Z[T]$ are $\Q$-linearly independent. Assume that $D = \max_{1 \le i \le M} \deg G_i \ge 2$. We have 
\begin{multline*}
\#\{n \le x: P_{MD+1}(n)>q, ~ ~(\forall i) ~ g_i(n) \equiv b_i \pmod q\}\\
\sim \frac1{q^M} \#\{n \le x: P_{MD+1}(n)>q\} \sim \frac x{q^M} \quad\text{ as $x\to\infty$},
\end{multline*} 
uniformly in moduli $q \le (\log x)^K$ lying in $\Qgfam$, and in residue classes $b_1, \dots , b_M$ mod $q$. 
\end{thm}
Here we omit the possibility $D=1$, as in this case, the fact that $D \ge M$ forces $M=1$, putting us in the setting of Theorem \ref{thm:UnrestrictedInput}(i), where we already have complete uniformity in $q$. For squarefree moduli $q$, it turns out that a much weaker restriction on the inputs suffices: we need only assume that $n$ has at least twice as many prime factors (counted with multiplicity) exceeding $q$ as the number $M$ of additive functions considered. 
\begin{thm}\label{thm:restrictediput_squarefreeq}
Fix $K \ge 1$, $M \ge 2$ and let $g_1, \dots , g_M$ be additive functions defined by the polynomials $G_1, \dots , G_M$, such that $\{G_i'\}_{1 \le i \le M} \subset \Z[T]$ are $\Q$-linearly independent. We have 
\begin{multline*}
\#\{n \le x: P_{2M}(n)>q, ~ ~ (\forall i) ~ g_i(n) \equiv b_i \pmod q\}\\
\sim \frac1{q^M} \#\{n \le x: P_{2M}(n)>q\} \sim \frac x{q^M} \quad\text{ as $x\to\infty$},
\end{multline*}
uniformly in squarefree $q \le (\log x)^K$ lying in $\Qgfam$, and in residues $b_1, \dots , b_M$ mod $q$.
\end{thm}
Here, we omit the case $M=1$ as complete uniformity in squarefree $q \le (\log x)^K$ has already been attained in Theorem \ref{thm:UnrestrictedInput}(i). In subsection \ref{subsec:OptimalitySqfree_RestrictedInput}, we will show that the restriction $P_{2M}(n)>q$ is nearly optimal in the sense that it cannot be weakened to $P_{2M-3}(n)>q$ for any $M \ge 2$, and that for $M=2$, it cannot be weakened to $P_{2M-2}(n)>q$ either.

We now illustrate the necessity of our recurring linear independence hypothesis. It turns out that if the polynomials $\{G_i'\}_{i=1}^M$ are not assumed to be $\Q$-linearly independent, then the $M$ congruences $g_i(n) \equiv b_i \pmod q$ might degenerate to (at most) $M-1$ congruences for sufficiently many inputs $n$. As such, it is not possible to restore uniformity in moduli $q \le (\log x)^K$ no matter how many prime factors of our inputs $n$ we assume to be larger than $q$. Specifically, for any large integer $R$, we can always construct integers $b_1, \dots , b_M$ which are overrepresented by the $g_1, \dots , g_M$ among the set of inputs $n \le x$ having $P_R(n)>q$. We show this precisely below; in what follows, $P^-(q)$ denotes the smallest prime divisor of $q$.
\begin{thm}\label{thm:LIHNecess}
Fix $K \ge 1, M \ge 2$ and polynomials $G_1, \dots , G_{M-1} \in \Z[T]$ such that $\{G_i'\}_{i=1}^{M-1} \subset \Z[T]$ are $\Q$-linearly independent. Consider nonzero integers $\{a_i\}_{i=1}^{M-1}$ and a polynomial $G_M \in \Z[T]$ satisfying $G_M' = \sum_{i=1}^{M-1} a_i G_i'$ and $G_M(0) \ne \sum_{i=1}^{M-1} a_i G_i(0)$. Let $g_1, \dots , g_M$ be additive functions defined by the polynomials $G_1, \dots , G_M$. There exists a computable constant $\CGsub>0$ depending only on the system $\widehat{G} \coloneqq (G_1, \dots, G_M)$ that satisfies the following properties:

For any integer $Q>1$ with $P^-(Q)>\CGsub$, $g_1, \dots , g_M$ are jointly equidistributed mod $Q$. However, for any fixed $R>\CGsub$ and any integers $\{b_i\}_{i=1}^{M-1}$, there exists an integer $b_M$ such that
$$\#\{n \le x: P_R(n)>q, ~ ~ (\forall i) ~ g_i(n) \equiv b_i \pmod q\} \gg \frac{x (\log_2 x)^{R-1}}{q^{M-1}\log x}  \quad\text{ as $x\to\infty$},$$
uniformly in moduli $q \le (\log x)^K$ having $P^-(q)>\CGsub$. 
\end{thm}
Thus, the above theorem shows that without the $\Q$-linear independence of the $\{G_i'\}_{i=1}^M$, uniformity could fail to \textit{all} moduli $q > \log x$ having sufficiently large prime factors, despite $g_1, \dots , g_M$ being jointly equidistributed to any fixed modulus having sufficiently large prime factors. We expect that with appropriate modifications of our methods, it might be possible to obtain analogues of Theorems \ref{thm:UnrestrictedInput}, \ref{thm:restrictediput_generalq} and \ref{thm:restrictediput_squarefreeq} (with more limited ranges of unformity in $q$) when $\{G_i'\}_{i=1}^M$ are not $\Q$-linearly independent: from the arguments we shall give for our main results, it seems reasonable to expect that the corresponding ranges of $q$ and restrictions on the inputs $n$ should then depend on the rank of the matrix of coefficients of the polynomials $\{G_i'\}_{i=1}^{M}$. 

We conclude this introductory section with the remark that although for the sake of simplicity of statements, we have been assuming that our additive functions $\{g_i\}_{i=1}^M$ and polynomials $\{G_i\}_{i=1}^M$ are both fixed, our proofs of Theorems \ref{thm:UnrestrictedInput}, \ref{thm:restrictediput_generalq}, \ref{thm:restrictediput_squarefreeq} and \ref{thm:LIHNecess} will reveal that these results are also uniform in the additive functions $\{g_i\}_{i=1}^M$ as long as they are defined by the fixed polynomials $\{G_i\}_{i=1}^M$. 
\subsection*{Notation and conventions:} 
Given polynomials $G_1, \dots , G_M \in \Z[T]$, we shall always use $D$ and $\Dmin$ to denote the maximum and the minimum of the degrees of the $G_i$, respectively. As mentioned previously, we shall use $P_k(n)$ to denote the $k$-th largest prime factor of $n$ (counted with multiplicity), $P(n)$ to denote $P_1(n)$, and $P^-(n)$ to denote the least prime divisor of $n$. We denote the number of primes dividing $q$ counted with and without multiplicity by 
$\Omega(q)$ and $\omega(q)$ respectively, and we write $U_q$ to denote the group of units (or multiplicative group) modulo $q$, so that $\#U_q = \phi(q)$, the Euler totient of $q$. When there is no danger of confusion, we shall write $(a_1, \dots , a_k)$ in place of $\gcd(a_1, \dots , a_k)$. 

Throughout, the letters $p$ and $\ell$ are reserved for primes. Implied constants in $\ll$ and $O$-notation, as well as implicit constants in qualifiers like ``sufficiently large", may always depend on any parameters declared as ``fixed''; in particular, they will always depend on the polynomials $G_1, \dots , G_M$. Other dependence will be noted explicitly (for example, with parentheses or subscripts); notably, we shall use $\CG$ or $\CGsub$ to denote constants depending only on the vector $\widehat{G} \coloneqq (G_1, \dots , G_M)$ of defining polynomials. For a nonzero polynomial $H \in \Z[T]$, we use $\ord_\ell(H)$ to denote the highest power of $\ell$ dividing all the coefficients of $H$; for an integer $m \ne 0$, we shall sometimes use $v_\ell(m)$ in place of $\ord_\ell(m)$. For a positive integer $n$, we define $\OmStqn \coloneqq \largesum_{\substack{p^k \parallel n\\p>q, ~ k>1}} k$ to be the number of prime divisors of $n$ (counted with multiplicity) that exceed $q$ and do not exactly divide $n$ (that is, appear to an exponent greater than $1$ in the prime factorization of $n$). We write $\log_{k}$ for the $k$-th iterate of the natural logarithm. 
\section{Preliminary Discussion: Delange's equidistribution criteria and consequences for polynomially-defined additive functions}\label{sec:PrelimDiscuss}
The following result of Delange provides a characterization for when a single additive function is equidistributed to a given integer modulus (see Theorem 1 and Remark 3.1.1 in \cite{delange69}). 
\begin{prop}\label{prop:DelangeSingle}
Let $f$ be an integral-valued additive function and $q>1$ a given integer. Consider the sums $S_d \coloneqq \sum_{p: ~d \nmid f(p)} 1/p$. Then $f$ is equidistributed mod $q$ if and only if $S_\ell$ diverges for every odd prime $\ell$ dividing $q$ and one of the following hold:\\
(i) $q$ is odd;\\
(ii) $2 \parallel q$, and either $S_2$ diverges or $f(2^r)$ is odd for all $r \ge 1$;\\
(iii) $4\mid q$, $S_4$ diverges, and either $S_2$ diverges or $f(2^r)$ is odd for all $r \ge 1$.
\end{prop}
In his sequel \cite{delange74} to the aforementioned paper, Delange characterizes when a given family $f_1, \dots , f_M$ of integral-valued additive functions is jointly equidistributed to a given integer modulus $q$, by reducing the problem to the equidistribution of a single additive function. The following is the relevant special case of Delange's result (which corresponds to the assignment $q_i' \coloneqq 1$, $\delta \coloneqq q$ in the result stated in section 4 of \cite{delange74}).
\begin{prop}\label{prop:DelangeJoint}
A given family $f_1, \dots , f_M$ of integral-valued additive functions is jointly equidistributed modulo $q>1$ if and only if for all integers $k_1, \dots , k_M$ satisfying $\gcd(k_1, \dots , k_M)=1$,\footnote{Whenever we speak of $\gcd(k_1, \dots , k_M)$, we assume implicitly that $(k_1, \dots , k_M) \ne (0, \dots , 0)$.} the additive function $k_1 f_1 + \dots + k_M f_M$ is equidistributed mod $q$.
\end{prop}
We remark that the formulation above is equivalent to that in \cite[Section 4]{delange74}, which is stated with the additional restriction that $k_1, \dots , k_M \in \{0, \dots , q-1\}$. Indeed, assume that $\sum_{i=1}^M \lambda_i g_i$ is equidistributed mod $q$ for all $(\lambda_1, \dots , \lambda_M) \in \{0, 1 \cdots, q-1\}^M$ satisfying $\gcd(\lambda_1, \dots , \lambda_M)=1$. We claim that $\sum_{i=1}^M k_i g_i$ is equidistributed mod $q$ for all $(k_1, \dots , k_M) \in \Z^M$ satisfying $\gcd(k_1, \dots , k_M)=1$. To see this, we consider any tuple $(k_1, \dots , k_M) \in \Z^M$ having $\gcd(k_1, \dots , k_M)=1$, and let $k_1', \dots , k_M' \in \{0, 1, \dots , q-1\}$ be the unique integers satisfying $k_i' \equiv k_i \pmod q$. Then $d' \coloneqq \gcd(k_1', \dots , k_M') \in \{1, \dots , q-1\}$ must be coprime to $q$, for otherwise, there is a prime $\ell$ dividing $\gcd(q, k_1', \dots , k_M')$ hence also dividing $\gcd(q, k_1, \dots , k_M)$ $=1$. Write $k_i' \eqqcolon d' k_i''$ for some $k_1'', \dots , k_M'' \in \{0, 1, \dots , q-1\}$ having $\gcd(k_1'', \dots , k_M'')=1$. Since $d'$ is invertible mod $q$ and the function $\sum_{i=1}^M k_i'' g_i$ is equidistributed mod $q$, it follows so is the function $\sum_{i=1}^M k_i g_i$, as $\sum_{i=1}^M k_i g_i \equiv \sum_{i=1}^M k_i' g_i \equiv d'\sum_{i=1}^M k_i'' g_i \pmod q$.

Propositions \ref{prop:DelangeSingle} and \ref{prop:DelangeJoint} lead to the following consequences in our setting of polynomially-defined additive functions, which is how they shall be useful to us. In what follows, for a given polynomial $G \in \Z[T]$, we let 
$$\alpha_G(q) \coloneqq \frac1{\phi(q)}\#(G^{-1}(U_q) \cap U_q) = \frac1{\phi(q)} \#\{v \in U_q: G(v) \in U_q\}$$
denote the proportion of unit residues $v$ mod $q$ whose image under the polynomial $G$ is also a unit mod $q$. By the Chinese Remainder Theorem, we see that $\alpha_G(q) = \prod_{\ell \mid q} \alpha_G(\ell)$.
\begin{lem}\label{lem:DelangePolyDefSingle}
Let $g \colon \NatNos \rightarrow \Z$ be an additive function defined by a nonconstant polynomial $G \in \Z[T]$. We can describe the set $\Qg = \{q \in \NatNos: g \text{ is equidistributed mod }q\}$ as follows:
\begin{enumerate}
    \item[(i)] If $2 \mid g(2^r)\text{ for some }r \ge 1$, then $\Qg = \{q: \alpha_G(q) \ne 0\}$.
    \item[(ii)] If $2 \nmid g(2^r)\text{ for all }r \ge 1\text{ and if }4 \mid (G(1), G(3))$, then $$\Qg = \{q: 2 \nmid q, ~ \alpha_G(q) \ne 0\} \cup \{q: 2 \parallel q, ~ \alpha_G(q/2) \ne 0\}.$$ 
    \item[(iii)] If $2 \nmid g(2^r)\text{ for all }r \ge 1\text{ and if }4 \nmid (G(1), G(3))$, then $\Qg = \{q: \alpha_G(q/2^{v_2(q)}) \ne 0\}$.
\end{enumerate}
\end{lem}
\begin{proof}
In what follows, let $q' \coloneqq q/2^{v_2(q)}$ denote the largest odd divisor of $q$. An application of the Siegel--Walfisz Theorem with partial summation shows that for any divisor $d>1$ of $q$ and any $X>e^q$, we have
$$S_d(X) \coloneqq \largesum_{\substack{p \le X\\d \nmid g(p)}} \frac1p = \mathlarger{\sum}_{\substack{p \le X\\d \nmid G(p)}} \frac1p = \mathlarger{\sum}_{\substack{r \in U_d\\d \nmid G(r)}} ~ \mathlarger{\sum}_{\substack{p \le X\\p \equiv r \bmod d}} \frac1p + O_q(1) = \beta_G(d) \log_2 X + O_q(1),$$
where $\beta_G(d) \coloneqq \frac1{\phi(d)} \#\{r \in U_d: d \nmid G(r)\}.$ Letting $X \rightarrow \infty$, we deduce that the sum $S_d = \sum_{p: ~ d \nmid g(p)} 1/p$ diverges if and only if $\beta_G(d) \ne 0$. But since $\beta_G(\ell) = \alpha_G(\ell)$ for any prime $\ell$, Proposition \ref{prop:DelangeSingle} shows that if $q \in \mathcal Q_g$, then $\alpha_G(\ell) \ne 0$ for all odd primes $\ell$ dividing $q$, so that $\alpha_G(q') \ne 0$. On the other hand, if $\alpha_G(q) \ne 0$ for some $q$, then $\beta_G(\ell) = \alpha_G(\ell) \ne 0$ for all primes dividing $q$, so that $S_\ell$ diverges for all such primes, and Proposition \ref{prop:DelangeSingle} leads to $q \in \Qg$ (since $S_4 \ge S_2$). In summary, we have so far shown that $\{q: \alpha_G(q) \ne 0\} \subset \Qg \subset \{q: \alpha_G(q') \ne 0\}$, which in particular means that $\{q: 2 \nmid q, ~ q \in \Qg\} = \{q: 2 \nmid q, ~ \alpha_G(q) \ne 0\}$. 

Now consider an even integer $q \in \Qg$, so that it satisfies the necessary condition $\alpha_G(q') \ne 0$.
\begin{enumerate}
\item[(i)] If $2 \mid g(2^r)\text{ for some }r \ge 1$, then by Proposition \ref{prop:DelangeSingle}, the sum $S_2$ must diverge. By the above discussion, this means that $\alpha_G(2) = \beta_G(2)$ must be nonzero, leading to $\alpha_G(q) \ne 0$. Hence, in this case $\Qg = \{q: \alpha_G(q) \ne 0\}$. 
\item[(ii)] Suppose $2 \nmid g(2^r)\text{ for all }r \ge 1\text{ and }4 \mid (G(1), G(3))$. Then $\alpha_G(2) = 0$, so that by Proposition \ref{prop:DelangeSingle}(ii) and the discussion in the previous paragraph, we have $\{q: 2 \parallel q, ~ q \in \Qg\} = \{q: 2 \parallel q, ~ \alpha_G(q/2) \ne 0\}$. Moreover, no positive integer divisible by $4$ can lie in $\Qg$: this follows by Proposition \ref{prop:DelangeSingle}(iii), since the condition $4 \mid (G(1), G(3))$ implies that $\beta_G(4) = 0$, and that $S_4$ converges. Hence, in this case $\Qg$ is as in the statement of the lemma.
\item[(iii)] Finally if $2 \nmid g(2^r)\text{ for all }r \ge 1\text{ and if }4 \nmid (G(1), G(3))$, then $S_4$ diverges, and Proposition \ref{prop:DelangeSingle} along with the inclusions obtained in the previous paragraph show that $q$ lies in $\Qg$ if and only if $\alpha_G(q') \ne 0$. 
\end{enumerate}
This completes the proof of the lemma.
\end{proof}
The following observation paves the way for a simple application of Proposition \ref{prop:DelangeJoint} in the setting of polynomially-defined additive functions. 
\begin{lem}\label{lem:alphaNonzeroLI}
Let $M \ge 2$ and $g_1, \dots , g_M: \NatNos \rightarrow \Z$ be additive functions defined by the nonconstant polynomials $G_1, \dots , G_M \in \Z[T]$, and let $\ell$ be a prime. If $\alpha_{k_1 G_1 + \dots + k_M G_M}(\ell) \ne 0$ for all integer tuples $(k_1, \dots , k_M)$ satisfying $\gcd(k_1, \dots , k_M)=1$, then the polynomials $G_1, \dots , G_M$ must be $\F_\ell$-linearly independent. Further, if $\ell>D+1$, then this condition is also sufficient.
\end{lem}
\begin{proof}
To establish the first assertion, we assume by way of contradiction that there exist $\mu_1, \dots , \mu_M \in \{0, 1, \dots , \ell-1\}$ not all zero, such that $\sum_{r=1}^M \mu_r G_r(T)$ vanishes identically in $\F_\ell[T]$. We will construct integers $k_1, \dots , k_M$ satisfying $\gcd(k_1, \dots , k_M) = 1$ and $\alpha_{k_1 G_1 + \dots + k_M G_M}(\ell)=0$. To that end, consider some $i \in [M]$ for which $\mu_i \not\equiv 0 \pmod\ell$ and let $k_r \coloneqq \mu_r$ for all $r \in [M]\sm\{i\}$. 

Now choose any $j \in [M]\sm\{i\}$. By the Chinese Remainder Theorem, there exists an integer $k_i$ such that $k_i \equiv \mu_i \pmod \ell$ and $\gcd(k_i, k_j)=1$. With this choice of integers $(k_1, \dots , k_M)$, we see that $\gcd(k_1,\cdots, k_M)=1$ and that the polynomial $\sum_{r=1}^M k_r G_r(T) \equiv \sum_{r=1}^M \mu_r G_r(T) \pmod \ell$ is identically zero in $\F_\ell[T]$, so that $\alpha_{k_1 G_1 + \dots + k_M G_M}(\ell)=0$. This proves the first assertion of the lemma. 

To show the second assertion, we consider any prime $\ell>D+1$. Suppose there did exist a tuple of integers $(k_1, \dots , k_M)$ satisfying $\gcd(k_1, \dots , k_M)=1$ and $\alpha_{k_1 G_1 + \dots + k_M G_M}(\ell)=0$. Then on the one hand, $(k_1, \dots , k_M) \not\equiv (0, \dots , 0) \pmod \ell$. On the other hand, the polynomial $\sum_{r=1}^M k_r G_r(T)$ (considered as an element of $\F_\ell[T]$) has degree at most $D$ but has at least $\#U_\ell = \phi(\ell) = \ell-1 > D$ roots in $\F_\ell$. As such, $\sum_{r=1}^M k_r G_r(T)$ vanishes identically in $\F_\ell[T]$ yielding a nontrivial $\F_\ell$-linear dependence relation between the $\{G_r\}_{r=1}^M$.    
\end{proof}
We remark that the matrix of coefficients alluded to towards the end of the introduction will play a pivotal role in our arguments. To set things up, we write $G_i'(T) \eqqcolon \sum_{r=0}^{D-1} a_{i, r} T^r$ for some integers $\{a_{i, r}: 1 \le i \le M, 0 \le r \le D-1\}$, so that $a_{i, D-1} \ne 0$ for some $i$ (since $D = \max_{1 \le i \le M} \deg G_i$). Note that since $G_i \in \Z[T]$, we have $(r+1) \mid a_{i, r}$ for all $i \in [M]$ and $0 \le r \le D-1$. By the \textsf{matrix of coefficients} or \textsf{coefficient matrix of the polynomials $\GiPrSet$}, we shall mean the $D \times M$ integer matrix
\begin{equation}\label{eq:A0Def}
A_0 \coloneqq ~
\begin{pmatrix}
a_{1, 0} & \cdots & a_{M, 0}\\
\cdots & \cdots & \cdots\\
\cdots & \cdots & \cdots\\
a_{1, D-1} & \cdots & a_{M, D-1}
\end{pmatrix}
\end{equation}
whose the $i$-th column lists the coefficients of the polynomial $G_i'$ in ascending order of the degree of $T$. It is important to note that if the polynomials $\{G_i'\}_{i=1}^M$ are $\Q$-linearly independent, then the columns of the matrix $A_0$ form $\Q$-linearly independent vectors, so that $A_0$ has full rank. As such, the Smith normal form $S_0$ of $A_0$ only has nonzero entries on its main diagonal. In other words, $A_0$ has exactly $M$ invariant factors $\beta_1, \dots , \beta_M \in \Z \sm \{0\}$, which must also satisfy $\beta_i \mid \beta_{i+1}$ for all $1 \le i < M$. Furthermore, since $S_0$ is obtained from $A_0$ by a base change over $\Z$, it follows that the primes $\ell$ for which the columns of $A_0$  (or equivalently, the polynomials $\{G_i'\}_{i=1}^M$) are $\F_\ell$-linearly dependent are precisely those which divide at least one of the $\beta_i$ (or equivalently, those which divide $\beta_M$). As a consequence, letting $\COG$ be any constant exceeding $\max\{D+1, |\beta_M|\}$ (so that $\COG$ depends only on the vector $\widehat{G} \coloneqq (G_1, \dots , G_M)$), we see that: 
\begin{equation}\label{eq:COGDefn}
\text{The polynomials }\{G_i'\}_{i=1}^M\text{ are }\F_\ell\text{-linearly independent for all primes }\ell>\COG.
\end{equation}
Our arguments leading to \eqref{eq:COGDefn} show that under the weaker hypothesis that the polynomials $\{G_i\}_{i=1}^M$ are $\Q$-linearly independent, then there exists a constant $C_1(\widehat{G})>D+1$ such that $\{G_i\}_{i=1}^M$ are $\F_\ell$-linearly independent for all $\ell>C_1(\widehat{G})$. Note that if $\{G_i'\}_{i=1}^M$ are $\Q$ (respectively, $\F_\ell$)-linearly independent, then so are $\{G_i\}_{i=1}^M$. Hence, if $\{G_i'\}_{i=1}^M$ are $\Q$-linearly independent, then with $\COG$ as in \eqref{eq:COGDefn}, the $\{G_i\}_{i=1}^M$ are also $\F_\ell$-linearly independent for any prime $\ell>\COG$. Combining these observations with Proposition \ref{prop:DelangeJoint} and Lemmas \ref{lem:DelangePolyDefSingle} and \ref{lem:alphaNonzeroLI}, we obtain the following useful consequence.
\begin{cor}\label{cor:JtEquidLargePD}
Let $g_1, \dots , g_M: \NatNos \rightarrow \Z$ be additive functions defined by the nonconstant polynomials $G_1, \dots , G_M \in \Z[T]$. Then for any $q>1$ with $P^-(q)>D+1$, the functions $g_1, \dots , g_M$ are jointly equidistributed mod $q$ if and only if the polynomials $\{G_i\}_{i=1}^M$ are $\F_\ell$-linearly independent for every prime $\ell \mid q$. In particular,
\begin{enumerate}
\item[(i)] If the polynomials $\{G_i\}_{i=1}^M$ are $\Q$-linearly independent (so that $C_1(\widehat{G})$ exists), then any $q$ having $P^-(q)>C_1(\widehat{G})$ lies in $\Qgfam$.
\item[(ii)] If the polynomials 
$\{G_i'\}_{i=1}^M$ are $\Q$-linearly independent (so that $C_0(\widehat{G})$ exists), then any $q$ having $P^-(q)>\COG$ lies in $\Qgfam$.
\end{enumerate}
\end{cor} 
\section{Preparation for Theorems \ref{thm:UnrestrictedInput}, \ref{thm:restrictediput_generalq} and \ref{thm:restrictediput_squarefreeq}: Obtaining the main term}\label{sec:ConvnMainTerm}
We start by defining 
$$J \coloneqq J(x) \coloneqq \lfloor \log_3 x \rfloor.$$
Let $\delta \in (0, 1]$ be as in the statement of Theorem 1.1; the development in this section will also go through in Theorems \ref{thm:restrictediput_generalq} and \ref{thm:restrictediput_squarefreeq} with (say) $\delta \coloneqq 1$. We define 
$$y \coloneqq \exp\left((\log x)^{\delta/2}\right),$$
and call a positive integer $n \le x$ \textsf{convenenient} if the $J$ largest prime divisors of $n$ exceed $y$ and exactly divide $n$, that is, if
$$\max\{P_{J+1}(n), y\} < P_J(n) < \cdots < P_1(n).$$
Any convenient $n$ can thus be uniquely written in the form $mP_J \cdots P_1$, with
\begin{equation}\label{eq:convnform}
L_m \coloneqq \max\{y, P(m)\} < P_J < \cdots < P_1.
\end{equation}
We will show that the convenient $n$ give the most dominant contribution to the counts considered in Theorems \ref{thm:UnrestrictedInput}, \ref{thm:restrictediput_generalq} and \ref{thm:restrictediput_squarefreeq}. 
\begin{prop}\label{prop:convenientprop} 
Fix $K, M \ge 1$ and let $g_1, \dots , g_M$ be additive functions defined by the nonconstant polynomials $G_1, \dots , G_M \in \Z[T]$, such that $\{G_i'\}_{1 \le i \le M} \subset \Q[T]$ are $\Q$-linearly independent. Let  $D = \max_{1 \le i \le M} \deg G_i$. We have
$$\#\{n \le x:n\text{ convenient},~(\forall i)~ g_i(n) \equiv b_i \pmod q\} \sim \frac x{q^M}, \quad\text{ as $x\to\infty$},$$
uniformly in moduli $q \le (\log x)^K$ lying in $\Qgfam$, and in residues $b_1, \dots , b_M$ mod $q$. 
\end{prop}
\begin{proof}
Writing each convenient $n$ uniquely in the form $mP_J \cdots P_1$, where $m, P_J, \dots , P_1$ satisfy \eqref{eq:convnform}, we find that $g_i(n) = g_i(m) + \sum_{j=1}^J G_i(P_j)$. The conditions $g_i(n) \equiv b_i \pmod q$ ($1 \le i \le M$) can then be rewritten as $(P_1, \dots , P_J)$ mod $q$ $\in \VqmPr \coloneqq \VJMqbigim$, where 
$$\VJMqwi \coloneqq \left\{(v_1, \dots , v_J) \in (U_q)^J: ~ (\forall i) ~ \sum_{j=1}^J G_i(v_j) \equiv w_i \pmod q\right\}.$$
(Note that this set can be defined for any set of polynomials $\{G_i\}_{i=1}^M$ regardless of whether or not they come from a set of additive functions.) As a consequence, 
\begin{equation}\label{eq:ConvnInitialSplit}
\begin{split}
\mathlarger{\sum}_{\substack{n \le x\text{ convenient }\\(\forall i) ~ g_i(n) \equiv b_i \pmod q}} 1 ~ ~ &= \mathlarger{\sum}_{m \le x} ~ \mathlarger{\sum}_{(v_1, \dots , v_J) \in \VqmPr} \mathlarger{\sum}_{\substack{P_1, \dots , P_J\\P_1 \cdots P_J \le x/m\\L_m < P_J < \cdots < P_1\\(\forall j) ~ P_j \equiv v_j \pmod q}} 1\\
&= \mathlarger{\sum}_{m \le x} ~\mathlarger{\sum}_{(v_1, \dots , v_J) \in \VqmPr} \frac1{J!} \mathlarger{\sum}_{\substack{P_1, \dots , P_J > L_m\\P_1 \cdots P_J \le x/m\\ P_1, \dots , P_J \text{ distinct }\\(\forall j) ~ P_j \equiv v_j \pmod q}} 1,
\end{split}
\end{equation}
where in the last equality above, we have noted that the conditions $P_1 \cdots P_J \le x/m$ and $(P_1, \dots , P_J)$ mod $q$ $\in \VqmPr$ are both independent of the ordering of $P_1, \dots , P_J$. 

We now estimate the innermost sum on $P_1, \dots , P_J$ by removing the congruence conditions on the $P_j$. For each tuple $(v_1,\dots, v_J)\bmod{q} \in \VqmPr$, we see that
\[ \mathlarger{\sum}_{\substack{P_1, \dots , P_J > L_m\\P_1 \cdots P_J \le x/m\\ P_1, \dots , P_J \text{ distinct }\\(\forall j) ~ P_j \equiv v_j \pmod q}} 1 =  \mathlarger{\sum}_{\substack{P_2, \dots , P_J > L_m\\P_2 \cdots P_J \le x/mL_m\\ P_2, \dots , P_J \text{ distinct }\\(\forall j) ~ P_j \equiv v_j \pmod q}} \mathlarger{\sum}_{\substack{P_1 \ne P_2,\dots,P_J \\ L_m < P_1 \le x/mP_2\cdots P_J \\ P_1 \equiv v_1\pmod q}} 1.\] 

Since $L_m \ge y$ and $q\le (\log{x})^{K} = (\log{y})^{2K/\delta}$, the Siegel--Walfisz theorem \cite[Corollary 11.21]{MV07} yields
\[\mathlarger{\sum}_{\substack{P_1 \ne P_2,\dots,P_J \\ L_m < P_1 \le x/mP_2\cdots P_J \\ P_1 \equiv v_1\pmod q}} 1 = \frac{1}{\phi(q)} \mathlarger{\sum}_{\substack{P_1 \neq P_2,\dots,P_J \\ L_m < P_1 \le x/mP_2\cdots P_J}}1 + O\left(\frac{x}{mP_2\cdots P_J} \exp(-C_0 \sqrt{\log y})\right),\] 
for some positive constant $C_0 := C_0(K, \delta)$ depending only on $K$ and $\delta$. Putting this back into the last display, we find that
\[ \mathlarger{\sum}_{\substack{P_1, \dots , P_J > L_m\\P_1 \cdots P_J \le x/m\\ P_1, \dots , P_J \text{ distinct }\\(\forall j) ~ P_j \equiv v_j \pmod q}} 1 ~ = ~  \frac1{\phi(q)}\mathlarger{\sum}_{\substack{P_1, \dots , P_J > L_m\\P_1 \cdots P_J \le x/m\\ P_1, \dots , P_J \text{ distinct }\\(\forall j \ge 2) ~ P_j \equiv v_j \pmod q}} 1 ~ + ~ O\left(\frac xm \exp\left(-\frac12 C_0 (\log{x})^{\delta/4}\right)\right),  \]
where we have put the bound 
$$\mathlarger{\sum}_{P_2, \dots , P_J \le x} \frac1{P_2 \cdots P_J} \le \left(\mathlarger{\sum}_{p \le x} \frac1p\right)^{J-1} \le (2\log_2 x)^{J-1} \le \exp(O((\log_3 x)^2)).$$
Proceeding in the same way to successively remove the congruence conditions on $P_2,\dots,P_J$, we deduce that
\begin{equation}\label{eq:CongrueceRemoved}
\mathlarger{\sum}_{\substack{P_1, \dots , P_J > L_m\\P_1 \cdots P_J \le x/m\\ P_1, \dots , P_J \text{ distinct }\\(\forall j) ~ P_j \equiv v_j \pmod q}} 1 ~ = ~  \frac1{\phi(q)^J}\mathlarger{\sum}_{\substack{P_1, \dots , P_J > L_m\\P_1 \cdots P_J \le x/m\\ P_1, \dots , P_J \text{ distinct }}} 1 ~ + ~ O\left(\frac xm \exp\left(-\frac14 C_0 (\log{x})^{\delta/4}\right)\right).
\end{equation}
Inserting this estimate into \eqref{eq:ConvnInitialSplit} and noting that $\#\VqmPr \le \phi(q)^J \le (\log x)^{KJ}$, we obtain
\begin{equation}\label{eq:convnSplitForm}
\mathlarger{\sum}_{\substack{n \le x\text{ convenient }\\(\forall i) ~ g_i(n) \equiv b_i \pmod q}} 1 = \mathlarger{\sum}_{m \le x} \frac{\#\VqmPr}{\phi(q)^J} \Bigg(\frac1{J!}\mathlarger{\sum}_{\substack{P_1, \dots , P_J > L_m\\P_1 \cdots P_J \le x/m\\ P_1, \dots , P_J \text{ distinct }}} 1\Bigg) ~ + ~ O\left(x \exp\left(-\frac18 C_0 (\log{x})^{\delta/4}\right)\right).     
\end{equation}
The following proposition, which we shall establish momentarily, will provide the desired estimate on the cardinalities of the sets $\VqmPr$. For future convenience and independent interest, we state it in slightly greater generality than necessary in our immediate application. 
\begin{prop}\label{prop:Vqwi_ReducnToBddModulus}
Let $G_1, \dots , G_M \in \Z[T]$ be nonconstant polynomials, such that $\{G_i'\}_{1 \le i \le M} \subset \Z[T]$ are $\Q$-linearly independent. Let $D = \max_{1 \le i \le M} \deg G_i$ and $C \coloneqq \CG$ be a constant exceeding $\max\{\COG, (2D)^{2D+4}\}$, where $\COG$ is the constant in \eqref{eq:COGDefn}. We have
\begin{multline*}
\frac{\#\mathcal V_{N, M}\left(q; (w_i)_{i=1}^M\right)}{\phi(q)^N}\\ = \left(\frac{Q_0}q\right)^M \left\{\frac{\#\VNMQZwi}{\phi(Q_0)^N} + O\left(\frac1{C^N}\right)\right\} \mathlarger{\prod}_{\substack{\ell \mid q\\\ell>C}} \left(1+O\left(\frac{(2D)^N}{\ell^{N/D-M}}\right)\right), 
\end{multline*}
uniformly in $N \ge MD+1$, in all positive integers $q>1$, and in residue classes $w_1, \dots , w_M$ mod $q$, where $Q_0$ is a divisor of $q$ of size $O(1)$ supported on primes at most $C$. 
\end{prop}
To estimate the count $\#\VqmPr$ in \eqref{eq:convnSplitForm}, we apply the above proposition with $N \coloneqq J$ which goes to infinity with $x$ and hence exceeds $MD+1$ for all sufficiently large $x$. For the same reason, we find that as $x \rightarrow \infty$,
$$\mathlarger{\sum}_{\substack{\ell \mid q\\\ell>C}} \frac{(2D)^N}{\ell^{N/D-M}} \le (2D)^J \mathlarger{\sum}_{\substack{\ell \mid q\\\ell>C}} \frac1{\ell^{J/(D+2)}} \le \frac{(2D)^J}{C^{J/(2D+4)}} \mathlarger{\sum}_{\ell \ge 2} \frac1{\ell^2} \le \left(\frac{2D}{C^{1/(2D+4)}}\right)^J = o(1).$$
As such, an application of the above proposition yields
$$\frac{\#\mathcal V_{J, M}\left(q; (w_i)_{i=1}^M\right)}{\phi(q)^J} = (1+o(1)) \left(\frac{Q_0}q\right)^M \left\{\frac{\#\VJMQZwi}{\phi(Q_0)^J} + O\left(\frac1{C^J}\right)\right\},$$
uniformly in $q$ and $(w_1, \dots , w_M)$ mod $q$, where $Q_0\mid q$ and $Q_0 = O(1)$. In particular, this same estimate holds for $\VqmPr = \VJMqbigim$, and we obtain from \eqref{eq:convnSplitForm},
\begin{align*}
\mathlarger{\sum}_{\substack{n \le x\text{ convenient }\\(\forall i) ~ g_i(n) \equiv b_i \pmod q}} 1 &= 
\begin{multlined}[t](1+o(1)) \left(\frac{Q_0}q\right)^M \mathlarger{\sum}_{m \le x} \left\{\frac{\#\VQZmPr}{\phi(Q_0)^J} + O(C^{-J})\right\} \Bigg(\frac1{J!}\mathlarger{\sum}_{\substack{P_1, \dots , P_J > L_m\\P_1 \cdots P_J \le x/m\\ P_1, \dots , P_J \text{ distinct }}} 1\Bigg)\\ + ~ O\left(x \exp\left(-\frac18 C_0 (\log{x})^{\delta/4}\right)\right)
\end{multlined}\\
&= (1+o(1)) \left(\frac{Q_0}q\right)^M \mathlarger{\sum}_{m \le x} \frac{\#\VQZmPr}{\phi(Q_0)^J} \Bigg(\frac1{J!}\mathlarger{\sum}_{\substack{P_1, \dots , P_J > L_m\\P_1 \cdots P_J \le x/m\\ P_1, \dots , P_J \text{ distinct }}} 1\Bigg) +  o\left(\frac x{q^M}\right)
\end{align*}
where we have recalled that 
$$\mathlarger{\sum}_{m \le x} \Bigg(\frac1{J!}\mathlarger{\sum}_{\substack{P_1, \dots , P_J > L_m\\P_1 \cdots P_J \le x/m\\ P_1, \dots , P_J \text{ distinct }}} 1\Bigg) \le \mathlarger{\sum}_{m \le x} \Bigg(\mathlarger{\sum}_{\substack{P_1, \dots , P_J\\P_1 \cdots P_J \le x/m\\ L_m < P_J< \cdots < P_1}} 1\Bigg) \le x.$$
But now, applying the estimate \eqref{eq:convnSplitForm} with $Q_0$ playing the role of $q$, we find that 
\begin{equation*}
\mathlarger{\sum}_{\substack{n \le x\text{ convenient }\\(\forall i) ~ g_i(n) \equiv b_i \pmod q}} 1 ~ = ~ (1+o(1)) \left(\frac{Q_0}q\right)^M \mathlarger{\sum}_{\substack{n \le x\text{ convenient }\\(\forall i) ~ g_i(n) \equiv b_i \pmod{Q_0}}} 1 ~ + ~ o\left(\frac x{q^M}\right).
\end{equation*}
Recall that any inconvenient $n \le x$ either has $P_J(n) \le y$ or has a repeated prime factor exceeding $y$. The number of $n \le x$ satisfying the latter condition is no more than $\sum_{p>y} \sum_{n \le x: ~p^2\mid n} 1 \le x\sum_{p>y} 1/p^2 \ll x/y = o(x)$. Moreover, by \cite[Lemma 2.3]{PSR22}, the number of $n \le x$ having $P_J(n) \le y$ is $\ll x(\log_2 x)^{J-1}/(\log x)^{1-\delta}$ which is also $o(x)$. This yields 
$$\mathlarger{\sum}_{\substack{n \le x \text{ convenient }\\(\forall i) ~ g_i(n) \equiv b_i \pmod q}} 1 ~ = ~ (1+o(1)) \left(\frac{Q_0}q\right)^M \mathlarger{\sum}_{\substack{n \le x\\(\forall i) ~ g_i(n) \equiv b_i \pmod{Q_0}}} 1 ~ + ~ o\left(\frac x{q^M}\right).$$
Finally, since $q$ lies in $\Qgfam$, so does its divisor $Q_0$, and as $Q_0 = O(1)$, the sum occurring on the right hand side above is $(1+o(1)) x/Q_0^M$. This completes the proof of Proposition \ref{prop:convenientprop}, up to that of Proposition \ref{prop:Vqwi_ReducnToBddModulus}.  
\end{proof}
Before beginning the proof of Proposition \ref{prop:Vqwi_ReducnToBddModulus}, we state some (relevant special cases of) known bounds on mixed exponential sums, which will provide some key technical inputs in our arguments. First, we have the renowned bound of Weil \cite{weil48} coming from his work on the Riemann Hypothesis for curves over a finite field (see also Schmidt \cite[chapter II, Corollary 2F]{schmidt76}). In what follows, we set $e(t) \coloneqq \exp(2 \pi i t)$. For a positive integer $Q$, we use $\chi_{0, Q}$ to denote the trivial (or principal) character mod $Q$. For a prime $\ell$, $\chi_{0, \ell}$ is also the principal character modulo any power of $\ell$. 
\begin{prop}\label{prop:Weil}
Let $F \in \Z[T]$ be a polynomial of degree $D_0 \ge 1$, and let $\ell>D_0$ be a prime such that $F$ doesn't reduce to a constant modulo $\ell$. Then we have $$\left|\sum_{v \bmod \ell} \chi_{0, \ell}(v) e(F(v)/\ell) \right| \le D_0 \ell^{1/2}.$$   
\end{prop}
We will also need analogues of the above bound for prime powers, which have been obtained by Cochrane and Zheng \cite[equation (1.13), Theorems 1.1 and 8.1]{CZ99}. (See \cite{cochrane02} for more general results.) In what follows, for a nonconstant polynomial $F \in \Z[T]$ and a prime $\ell$, we define $t_\ell(F) \coloneqq \ordellFPr$, that is $t_\ell(F)$ is the highest power of $\ell$ dividing the coefficients of the polynomial $F'$. Let $\critsetAFell$ denote the set of nonzero roots in $\Fell$ of the polynomial $\ell^{-t_\ell(F)} F'$ (considered as a nonzero element of $\FellT$). We use $M_\ell(F)$ to denote the maximum of the multiplicities of the zeros of $\ell^{-t_\ell(F)} F'$ in $\Fell$, with $M_\ell(F) \coloneqq \infty$ if there is no such zero.
\begin{prop}\label{prop:CZ}
Let $F \in \Z[T]$ be a polynomial of degree $D_0 \ge 1$, and let $\ell^e$ be a prime power such that $F$ doesn't reduce to a constant modulo $\ell$. Let $t \coloneqq t_\ell(F)$ and $M \coloneqq M_\ell(F)$.
\begin{enumerate}
\item[(i)] If $\ell>2$ and $e \ge t + 2$, then $$\left|\mathlarger{\sum}_{v \bmod \ell^e} \chiZell(v) e(F(v)/\ell^e)\right| \le D_0 \cdot \ell^{t/(M+1)} \cdot \ell^{e(1-1/(M+1))}.$$
\item[(ii)] For $\ell=2$ and $e \ge t + 3$, we have $$\left|\mathlarger{\sum}_{v \bmod 2^e} \chi_{0, 2}(v) e(F(v)/2^e)\right| \le 2D_0 \cdot 2^{t/(M+1)} \cdot 2^{e(1-1/(M+1))}.$$
\end{enumerate} 
\end{prop}
\begin{proof}[Proof of Proposition \ref{prop:Vqwi_ReducnToBddModulus}] We start by showing that 
\begin{equation}\label{eq:Largeprime}
\#\mathcal V_{N, M}\left(\ell^e; (w_i)_{i=1}^M\right) = \frac{\phi(\ell^e)^N}{\ell^{eM}} \left(1+O\left(\frac{(2D)^N}{\ell^{N/D-M}}\right)\right) 
\end{equation}
uniformly for all primes $\ell>C = \CG$, positive integers $e \ge 1$ and $N \ge MD+1$, and $w_i \in \Z/\ell^e\Z$. Indeed, by the orthogonality of additive characters, we can write
\begin{equation}\label{eq:primepowerOrth}
\begin{split}
\#\mathcal V_{N, M}\left(\ell^e; (w_i)_{i=1}^M\right) &= \#\left\{(v_1, \dots , v_N) \in (U_{\ell^e})^N: ~ (\forall i) ~ \sum_{j=1}^N G_i(v_j) \equiv w_i \pmod{\ell^e}\right\}\\
&= \mathlarger{\sum}_{(v_1, \dots , v_N) \in (U_{\ell^e})^N} \mathlarger{\prod}_{i=1}^M \left(\frac1{\ell^e} \mathlarger{\sum}_{r_i \bmod \ell^e} e\left(-\frac{r_i w_i}{\ell^e}\right) ~ e\left(\frac{r_i}{\ell^e} \sum_{j=1}^N G_i(v_j)\right)\right)\\
&= \frac{\phi(\ell^e)^N}{\ell^{eM}} \left\{1+\frac1{\phi(\ell^e)^N} \largesum_{(r_1, \dots , r_M) \not\equiv (0, \dots , 0) \bmod \ell^e} e\left(-\frac1{\ell^e}\sum_{i=1}^M r_i w_i\right) \left(Z_{\ell^e; \, r_1, \dots, r_M}\right)^N\right\},
\end{split}
\end{equation}
where $Z_{\ell^e; \, r_1, \dots, r_M} \coloneqq \largesum_{v \bmod \ell^e} \chi_{0, \ell}(v) e\left(\frac1{\ell^e}\largesum_{i=1}^M r_i G_i(v)\right)$ and $\chi_{0, \ell}$ denotes the trivial character mod $\ell^e$ (which is also the trivial character mod $\ell$). Now in the case $D=1$, we must have $M=1$, so that we may write $G_1(T) \eqqcolon AT+B$ for some integers $A \ne 0$ and $B$. For each nonzero residue $r$ mod $\ell^e$, we have $r \eqqcolon \ell^{e-e_0} r'$ for some $e_0 \in \{1,  \cdots, e\}$ and some coprime residue $r'$ mod $\ell^{e_0}$. Hence, $|Z_{\ell^e; \, r}| = \ell^{e-e_0} \left|\sum_{\substack{v\bmod \ell^{e_0}\\\gcd(v, \ell^{e_0})=1}} e(r'Av/\ell^{e_0})\right|$. The last sum being a Ramanujan sum is nonzero precisely when $\ell^{e_0-1} | r'A$ (see properties of Ramanujan sums in \cite{HW08} and \cite{MV07}). But this forces $e_0=1$ because $\ell \nmid A$ (by definition of $\COG = C_0(\{G_1\})$) and $\ell \nmid r'$ (by definition of $r'$.) If $e_0=1$, then $|Z_{\ell^e; \, r}| \le \ell^{e-1}$, and since there are at most $\ell$ many residues $r$ mod $\ell^e$ which are divisible by $\ell^{e-1}$, we find from \eqref{eq:primepowerOrth} that
$$\#\mathcal V_{N, M}\left(\ell^e; (w_i)_{i=1}^M\right) = \frac{\phi(\ell^e)^N}{\ell^{e}} \left\{1+O\left(\frac1{\phi(\ell^e)^N} \cdot \ell \cdot (\ell^{e-1})^N \right)\right\} = \frac{\phi(\ell^e)^N}{\ell^{e}} \left\{1+O\left(\frac{2^N}{\ell^{N-1}} \right)\right\}$$
uniformly in $N \ge 1$. This establishes the bound \eqref{eq:Largeprime} in the case $D=1$, so in order to complete the proof of \eqref{eq:Largeprime}, we may assume that $D \ge 2$.

Now for a given tuple $(r_1, \dots , r_M) \not\equiv (0, \dots , 0)$ mod $\ell^e$, we must have $\gcd(\ell^e, r_1, \dots , r_M) = \ell^{e-e_0}$ for some $1 \le e_0 \le e$. Hence, we can write $r_i \coloneqq \ell^{e-e_0} r_i'$ for some $(r_1', \dots , r_M')$ mod $\ell^{e_0}$ satisfying $(r_1', \dots , r_M') \not\equiv (0, \dots , 0) \bmod \ell$, which shows that
$$\left|Z_{\ell^e; \, r_1, \dots, r_M}\right| = \ell^{e-e_0}\left|\largesum_{v \bmod \ell^{e_0}} \chi_{0, \ell}(v) e\left(\frac1{\ell^{e_0}}\largesum_{i=1}^M r_i' G_i(v)\right)\right| = \ell^{e-e_0}\left|\largesum_{v \bmod \ell^{e_0}} \chi_{0, \ell}(v) e\left(\frac{F(v)}{\ell^{e_0}}\right)\right|,$$ where $F(T) \coloneqq \sum_{i=1}^M r_i' (G_i(T)-G_i(0))$. Now we observe that since $\ell>\CG>\COG$, the polynomials $\{G_i'\}_{i=1}^M$ are $\F_\ell$-linearly independent, hence so are the polynomials $\{G_i-G_i(0)\}_{i=1}^M$. This prevents the polynomial $F$ from reducing to a constant mod $\ell$ (for if it did, then this constant would be zero). Consequently, if $e_0=1$, then Proposition \ref{prop:Weil} yields $\left|Z_{\ell^e; \, r_1, \dots, r_M}\right| \le \ell^{e-e_0} \cdot D \ell^{1/2} = D\ell^{e-1/2}$. On the other hand, if $e_0 \ge 2$, then from Proposition \ref{prop:CZ}(i), we obtain $\left|Z_{\ell^e; \, r_1, \dots, r_M}\right| \le \ell^{e-e_0} \cdot D \ell^{e_0(1-1/D)}= D \ell^{e-e_0/D}$; here we have noted that $\ell>C>2$, $\tellF=\ord_\ell(F') = \ordell\big(\sum_{i=1}^M r_i' G_i'\big) =0 \le e_0-2$ and that $M_\ell(F) \le \deg(F') \le D-1$. For each $1 \le e_0 \le e$, there are at most $\ell^{e_0M}$ many possible tuples $(r_1', \dots , r_M')$ mod $\ell^{e_0}$, hence at most $\ell^{e_0M}$ tuples $(r_1, \dots , r_M)$ mod $\ell^e$ satisfying $\gcd(\ell^e, r_1, \dots , r_M) = \ell^{e-e_0}$. We deduce that 
\begin{equation*}
\begin{split}
\largesum_{(r_1, \dots , r_M) \not\equiv (0, \dots , 0) \bmod \ell^e} & \left|Z_{\ell^e; \, r_1, \dots, r_M}\right|^N \le \ell^{M} \left(D \ell^{e-1/2}\right)^N + \largesum_{2 \le e_0 \le e} \ell^{e_0M} \left(D \ell^{e-e_0/D}\right)^N \\ &\le \largesum_{1 \le e_0 \le e} \ell^{e_0M} \left(D \ell^{e-e_0/D}\right)^N \le \frac{D^N \ell^{eN}}{\ell^{N/D-M}} \largesum_{r \ge 0} \frac1{\left(\ell^{N/D-M}\right)^r} \ll \frac{D^N \ell^{eN}}{\ell^{N/D-M}},
\end{split}
\end{equation*}   
where the last bound uses the fact that $N/D-M \ge 1/D$, so that the last sum occurring in the above display is no more than $\largesum_{r \ge 0} 
2^{-r/D} \ll 1$. (It is while passing from the first line to the second in the above display where we use the assumption that $D \ge 2$.) Inserting the bound obtained above into \eqref{eq:primepowerOrth} and noting that $\ell/(\ell-1) \le 2$ completes the proof of estimate \eqref{eq:Largeprime}.

Given an arbitrary positive integer $q$, let $\qtil \coloneqq \prod_{\substack{\ell^e \parallel q\\\ell \le C}} \ell^e$ denote the largest divisor of $q$ supported on primes not exceeding the constant $C$ (the ``$C$-smooth part" of $q$). We can again invoke the orthogonality of additive characters to write, for any tuple of residues $(w_1, \dots , w_M)$ mod $\qtil$, 
\begin{equation}\label{eq:qtilOrth}
\begin{split}
\#\mathcal V_{N, M}\left(\qtil; (w_i)_{i=1}^M\right) &= \#\left\{(v_1, \dots , v_N) \in (U_{\qtil})^N: ~ (\forall i) ~ \sum_{j=1}^N G_i(v_j) \equiv w_i \pmod{\qtil}\right\}\\
&= \frac1{\qtil^M} \largesum_{r_1, \dots , r_M \bmod \qtil} e\left(-\frac1{\qtil}\sum_{i=1}^M r_i w_i\right) \left(Z_{\qtil; \, r_1, \dots, r_M}\right)^N,
\end{split}
\end{equation}
where $Z_{\qtil; \, r_1, \dots, r_M} \coloneqq \largesum_{v \bmod \qtil} \chiZqtil(v) e\left(\frac1{\qtil}\largesum_{i=1}^M r_i G_i(v)\right)$ and $\chiZqtil$ denotes the trivial character mod $\qtil$. 

Now with $\beta_1, \dots , \beta_M$ being the invariant factors of the matrix $A_0$ defined in \eqref{eq:A0Def} (listed in ascending order), we fix $R \coloneqq R(\widehat{G}) \in \NatNos_{\ge 2}$ to be any integer constant such that 
$$R> CD (4D |\beta_M|)^C.$$ Let $Q_1 \coloneqq \prod_{\ell^e \parallel \qtil: ~ e>R} \ell^{e-R}$ and $Q_0 \coloneqq \qtil/Q_1 = \prod_{\ell^e \parallel \qtil} \ell^{\min\{e, R\}} = \prod_{\ell^e \parallel q: ~\ell \le C} \ell^{\min\{e, R\}}$, so that $Q_0\mid q$ and $Q_0 \le \prod_{\ell \le C} \ell^R \ll 1$. We write $\#\mathcal V_{N, M}\left(\qtil; (w_i)_{i=1}^M\right) \eqqcolon S' + S''$, where $S'$ counts the contribution of all tuples $(r_1, \dots , r_M)$ mod $\qtil$ where all the components $r_i$ are divisible by $Q_1$, that is, 
$$S' \coloneqq \frac1{\qtil^M} \largesum_{\substack{r_1, \dots , r_M \bmod \qtil\\ (r_1, \dots , r_M) \equiv (0, \dots , 0) \bmod{Q_1}}} e\left(-\frac1{\qtil}\sum_{i=1}^M r_i w_i\right) \left(Z_{\qtil; \, r_1, \dots, r_M}\right)^N.$$ 
Any tuple $(r_1, \dots , r_M)$ mod $\qtil$ counted in $S'$ is thus of the form $(Q_1 s_1, \dots , Q_1 s_M)$ for some tuple $(s_1, \dots , s_M)$ mod $Q_0$ that is uniquely determined by $(r_1, \dots , r_M)$. We find that 
\begin{equation*}
\begin{split}
Z_{\qtil; \, r_1, \dots, r_M} &= \largesum_{v \bmod \qtil} \chiZqtil(v) e\left(\frac1{Q_0}\largesum_{i=1}^M s_i G_i(v)\right)\\ &= \largesum_{u \bmod Q_0} \chi_{0, Q_0}(u) e\left(\frac1{Q_0}\largesum_{i=1}^M s_i G_i(u)\right) \largesum_{\substack{v \in U_{\qtil}\\v \equiv u \bmod Q_0}} 1 = \frac{\phi(\qtil)}{\phi(Q_0)} ~ Z_{Q_0; \, s_1, \dots, s_M} 
\end{split}
\end{equation*}
where the last equality above follows from a simple counting argument. Consequently, 
\begin{equation*}
\begin{split}
S' = \frac1{\qtil^M} \left(\frac{\phi(\qtil)}{\phi(Q_0)}\right)^N \largesum_{s_1, \dots , s_M \bmod Q_0} e\left(-\frac1{Q_0}\sum_{i=1}^M s_i w_i\right) \left(Z_{Q_0; \, s_1, \dots, s_M}\right)^N.
\end{split}
\end{equation*}
An application of the orthogonality identity \eqref{eq:qtilOrth} with $Q_0$ playing the role of $\qtil$ yields 
\begin{equation}\label{eq:S'FinalForm}
S' = \left(\frac{Q_0}{\qtil}\right)^M \left(\frac{\phi(\qtil)}{\phi(Q_0)}\right)^N \#\mathcal V_{N, M}\left(Q_0; (w_i)_{i=1}^M\right).
\end{equation}
Now we consider the sum 
$$S'' = \frac1{\qtil^M} \largesum_{\substack{r_1, \dots , r_M \bmod \qtil\\ (r_1, \dots , r_M) \not\equiv (0, \dots , 0) \bmod Q_1}} e\left(-\frac1{\qtil}\sum_{i=1}^M r_i w_i\right) \left(Z_{\qtil; \, r_1, \dots, r_M}\right)^N.$$
Consider any tuple $(r_1, \dots , r_M)$ mod $\qtil$ occurring in $S''$. By the definition of $Q_1$, there exists a prime power $\ell^e \parallel \qtil$ for which $e>R$ but $v_\ell(\gcd(r_1, \dots, r_M))<e-R$. Letting $Q' \coloneqq \qtil/\gcd(\qtil, r_1, \dots, r_M)$ and $r_i' \coloneqq r_i/\gcd(\qtil, r_1, \dots, r_M)$ (for $1 \le i \le M$), we therefore deduce that for any such aforementioned prime $\ell$, we have $v_\ell(Q')>R$, so that $Q'$ is not $(R+1)$-free. Moreover, $r_1', \dots, r_M'$ are uniquely determined mod $Q'$ and satisfy $\gcd(Q', r_1', \dots, r_M')=1$. Now for each $i$, we can write $r_i'/Q' = \sum_{\ell^{e_\ell} \parallel Q'} r_{i, \ell}'/\ell^{e_\ell}$ mod $1$, where the sum is over the prime powers $\ell^{e_\ell}$ exactly dividing $Q'$; \footnote{We are just applying Bezout's identity; equivalently, this may be thought of as partial fraction decomposition over the integers.} here, for each $\ell^{e_\ell} \parallel Q'$, $r_{i, \ell}'$ is uniquely determined mod $\ell^{e_\ell}$ by the relation $r_{i, \ell}' \prod_{\substack{p^{e_p} \parallel Q'\\p \ne \ell}} p^{e_p} \equiv r_i' \pmod{\ell^{e_\ell}}$. Since $\gcd(Q', r_1', \dots, r_M')=1$, it follows that $\ell \nmid \gcd(r_{1, \ell}', \dots, r_{M, \ell}')$ for each prime $\ell \mid Q'$. By the Chinese Remainder Theorem, we can factor 
\begin{equation}\label{eq:CRTFactoring}
Z_{\qtil; \, r_1, \dots, r_M} = \frac{\phi(\qtil)}{\phi(Q')} \sum_{v \bmod Q'} \chi_{0, Q'}(v) e\left(\frac1{Q'} \sum_{i=1}^M r_i' G_i(v)\right) = \frac{\phi(\qtil)}{\phi(Q')} \prod_{\ell^{e_\ell} \parallel Q'} Z_{\ell^{e_\ell}; \, r_{1, \ell}', \dots, r_{M, \ell}'}.  
\end{equation} 
Write $G_i'(T) \eqqcolon \sum_{j=0}^{D-1} a_{i, j} T^j$ as in the discussion preceding \eqref{eq:A0Def}. We claim that for any prime $\ell \mid Q'$, 
\begin{equation}\label{eq:smallv_ell}
t_{\ell} \coloneqq t_\ell(r_{1, \ell}', \dots, r_{M, \ell}') \coloneqq \ord_\ell\left(\sum_{i=1}^M r_{i, \ell}' G_i'\right) = v_\ell\left(\gcd_{0 \le j \le D-1} \sum_{i=1}^M a_{i, j} r_{i, \ell}' \right) \le v_\ell(\beta_M),   
\end{equation}
where (recall) $\beta_1, \dots , \beta_M$ are the invariant factors of the matrix $A_0$ in \eqref{eq:A0Def}. The third equality simply follows from the fact that $\sum_{i=1}^M r_{i, \ell}' G_i'(T) = \sum_{j=0}^{D-1} \left(\sum_{i=1}^M a_{i, j} r_{i, \ell}'\right) T^j$. To show the inequality in \eqref{eq:smallv_ell}, it suffices to show that $\ell^{t_\ell}$ must divide $\beta_M$. To do the latter, we recall that, by the theory of modules over a principal ideal domain, that there exist a $D \times D$ integer matrix $P_0$ and an $M \times M$ integer matrix $R_0$ such that $\det P_0, \det R_0 \in \{\pm 1\}$ and $P_0 A_0 R_0$ is the Smith normal form $S_0$ of $A_0$. As such, $P_0 A_0 = S_0 R_0^{-1}$ where the matrix $R_0^{-1}$ has integer entries $(k_{i, j})_{1 \le i, j \le M}$. Now $\ell^{t_\ell}$ divides all the numbers $\{\sum_{i=1}^M a_{i, j} r_{i, \ell}' : 0 \le j \le D-1\}$, which are precisely the entries of the matrix $A_0 \begin{pmatrix}r_{1, \ell}' & \dots & r_{M, \ell}'\end{pmatrix}^\intercal$ (here $\begin{pmatrix}r_{1, \ell}' & \dots & r_{M, \ell}'\end{pmatrix}^\intercal$ denotes the column vector listing the $r_{i, \ell}'$). As such, $\ell^{t_\ell}$ also divides the entries of the matrix $P_0 A_0  \begin{pmatrix}r_{1, \ell}' & \dots & r_{M, \ell}'\end{pmatrix}^\intercal$, and hence also those of the matrix 
\begin{equation}\label{eq:SQcolumn}
S_0 R_0^{-1} \begin{pmatrix}r_{1, \ell}'\\ \dots \\ \dots \\r_{M, \ell}'\end{pmatrix}_{M \times 1} = \begin{pmatrix}\beta_1(k_{1, 1} r_{1, \ell}'  + \dots + k_{1, M} r_{M, \ell}' )\\ \dots \\ \dots \\ \beta_M(k_{M, 1} r_{1, \ell}'  + \dots + k_{M, M} r_{M, \ell}')\\ 0 \\ \dots \\ 0\end{pmatrix}_{D \times 1}.
\end{equation}
But now if $\ell$ divides all of the numbers $k_{1, 1} r_{1, \ell}'  + \dots + k_{1, M} r_{M, \ell}', \dots \dots, k_{M, 1} r_{1, \ell}'  + \dots + k_{M, M} r_{M, \ell}'$, then 
$$R_0^{-1} \begin{pmatrix}r_{1, \ell}'\\ \dots \\ \dots \\r_{M, \ell}'\end{pmatrix}_{M \times 1} = \begin{pmatrix} k_{1, 1} r_{1, \ell}'  + \dots + k_{1, M} r_{M, \ell}'\\ \cdots \\ \cdots \\ k_{M, 1} r_{1, \ell}'  + \dots + k_{M, M} r_{M, \ell}'\end{pmatrix}_{M\times 1} \equiv \begin{pmatrix}0 \\ \cdots \\ \cdots \\ 0\end{pmatrix}_{M\times 1} \pmod \ell.$$
This forces $\ell$ to divide $\gcd(r_{1, \ell}', \dots , r_{M, \ell}')$, which is impossible since $\ell \mid Q'$ (see the line preceding \eqref{eq:CRTFactoring}). Since $\ell^{t_\ell}$ divides the entries of the rightmost matrix in \eqref{eq:SQcolumn}, it follows that $\ell^{t_\ell}$ must divide at least one of the invariant factors $\beta_i$, and hence must also divide $\beta_M$. This establishes our claim \eqref{eq:smallv_ell}. 

We will now show that for any prime power $\ell^{e_\ell} \parallel Q'$ for which $e_\ell > R$, we have
\begin{equation}\label{eq:ZellellBound}
|Z_{\ell^{e_\ell}; \, r_{1, \ell}', \dots, r_{M, \ell}'}| = \left|\sum_{v \bmod \ell^{e_\ell}} \chi_{0, \ell}(v) e\left(\frac1{\ell^{e_\ell}} \sum_{i=1}^M r_{i, \ell}' G_i(v)\right)\right| \le 2D |\beta_M| \ell^{e_\ell(1-1/D)}.
\end{equation}
To show this, we note that since $G_i'(T) = \sum_{j=0}^{D-1} a_{i, j} T^j$, we have $G_i(T)-G_i(0) = \sum_{j=0}^{D-1} \frac{a_{i, j}}{j+1} T^{j+1}$ (recall that $(j+1) \mid a_{i, j}$), so that with 
\begin{equation}\label{eq:cEllDef}
\begin{split}
c_\ell \coloneqq \ord_\ell\left(\sum_{i=1}^M r_{i, \ell}' (G_i(T)-G_i(0))\right) 
= v_\ell\left(\gcd_{0 \le j \le D-1} \frac{\sum_{i=1}^M a_{i, j} r_{i, \ell}'}{j+1}\right),
\end{split}
\end{equation}
we have
\begin{multline*}
|Z_{\ell^{e_\ell}; \, r_{1, \ell}', \dots, r_{M, \ell}'}| = \left|\sum_{v \bmod \ell^{e_\ell}} \chi_{0, \ell}(v) e\left(\frac1{\ell^{e_\ell-c_\ell}} \sum_{j=0}^{D-1} \left(\ell^{-c_\ell}\frac{\sum_{i=1}^M a_{i, j} r_{i, \ell}'}{j+1} \right) v^{j+1}\right)\right|\\ = \ell^{c_\ell} \left|\sum_{v \bmod \ell^{e_\ell-c_\ell}} \chi_{0, \ell}(v) e\left(\frac{\Ftilv}{\ell^{e_\ell-c_\ell}}\right)\right|,
\end{multline*}
where $\FtilT \coloneqq \sum_{j=0}^{D-1} \left(\ell^{-c_\ell} \frac{\sum_{i=1}^M a_{i, j} r_{i, \ell}'}{j+1} \right) T^{j+1} \in \Z[T]$.  By \eqref{eq:cEllDef} and \eqref{eq:smallv_ell}, we see that $\Ftil$ cannot reduce to a constant mod $\ell$ and that $c_\ell \le t_\ell \le v_\ell(\beta_M)$. Furthermore, \eqref{eq:smallv_ell} also shows that $\ordell(\Ftil') = \ordell\left(\sum_{j=0}^{D-1} \left(\sum_{i=1}^M a_{i, j} r_{i, \ell}' \right) T^j\right) - c_\ell = t_\ell-c_\ell \le v_\ell(\beta_M) - c_\ell \le R-3-c_\ell < (e_\ell-c_\ell)-3$. (Here we use $e_\ell>R > |\beta_M|+3$.) Consequently, some subpart of Proposition \ref{prop:CZ} applies, yielding
\begin{equation*}
\begin{split}
|Z_{\ell^{e_\ell}; \, r_{1, \ell}', \dots, r_{M, \ell}'}| &\le \ell^{c_\ell} \cdot 2D \ell^{\ordell(\Ftil')} \cdot \ell^{(e_\ell - c_\ell)(1-1/(M_\ell(\Ftil)+1))}\\
&\le \ell^{c_\ell} \cdot 2D \ell^{v_\ell(\beta_M) - c_\ell} \cdot \ell^{e_\ell(1-1/D)} \le 2D |\beta_M| \ell^{e_\ell(1-1/D)}.
\end{split}
\end{equation*}
Here, $M_\ell(\Ftil)$ is the largest multiplicity of a zero in $\Fell$ of the polynomial $\ell^{-\ordell(\Ftil')} \Ftil'$, and we have used that this multiplicity is no more than $\deg(\Ftil') \le D-1$. This establishes \eqref{eq:ZellellBound}.

Applying the bound \eqref{eq:ZellellBound} to each prime power $\ell^{e_\ell} \parallel Q'$ for which $e_\ell>R$, and applying the trivial bound $|Z_{\ell^{e_\ell}; \, r_{1, \ell}', \dots, r_{M, \ell}'}| \le \phi(\ell^{e_\ell})$ for all the other prime powers $\ell^{e_\ell} \parallel Q'$, the factorization \eqref{eq:CRTFactoring} yields
\begin{equation*}
\begin{split}
|Z_{\qtil; \, r_1, \dots, r_M}| &\le \frac{\phi(\qtil)}{\phi(Q')} ~ \Bigg(\prod_{\substack{\ell^{e_\ell} \parallel Q'\\e_\ell \le R}} \phi(\ell^{e_\ell})\Bigg) \cdot \Bigg(\prod_{\substack{\ell^{e_\ell} \parallel Q'\\e_\ell > R}} 2D|\beta_M| \ell^{e_\ell(1-1/D)}\Bigg)\\
& \le (2D |\beta_M|)^{\omega(Q')} \cdot \phi(\qtil) \cdot \prod_{\substack{\ell^{e_\ell} \parallel Q'\\e_\ell > R}} \Bigg( \frac{\ell^{e_\ell(1-1/D)}}{\phi(\ell^{e_\ell})}\Bigg) \le (4D |\beta_M|)^C \cdot \frac{\phi(\qtil)}{A^{1/D}}.
\end{split}
\end{equation*}
Here $A$ denotes the $(R+1)$-full part of $Q'$ and in the last bound above, we have noted that $\omega(Q') \le \omega(\qtil) \le \sum_{\ell \le C} 1 \le C$. Note that since $Q'$ is not $(R+1)$-free, we have $A>1$. 

Applying this bound for each of the sums $Z_{\qtil; \, r_1, \dots, r_M}$ occurring in $S''$, we obtain
\begin{align}\allowdisplaybreaks
|S''| 
\le \frac{(4D |\beta_M|)^{CN} \phi(\qtil)^N}{\qtil^M} \largesum_{\substack{A\mid \qtil: ~ A>1\\A\text{ is }(R+1)\text{-full}}} \frac1{A^{N/D}} \largesum_{\substack{Q', r_1', \dots, r_M'\\Q' \mid \qtil: \, (R+1)\text{-full part of }Q'\text{ is }A\\ r_1', \dots , r_M' \bmod Q'\\ \gcd(r_1', \dots , r_M', Q')=1}} \largesum_{\substack{r_1, \dots , r_M \bmod \qtil\\Q' = \qtil/\gcd(\qtil, r_1, \dots , r_M)\\(\forall i)~ r_i' = r_i/\gcd(\qtil, r_1, \dots , r_M)}} 1. \nonumber
\end{align}
Since any choice of $Q'\mid \qtil$ and residues $r_1', \dots , r_M'$ mod $Q'$ uniquely determines $r_1, \dots , r_M$ mod $\qtil$ by the relations $r_i = r_i'\qtil/Q'$, we see that
\begin{align*}\allowdisplaybreaks
|S''| &\le \frac{(4D |\beta_M|)^{CN} \phi(\qtil)^N}{\qtil^M} \largesum_{\substack{A\mid \qtil: ~ A>1\\A\text{ is }(R+1)\text{-full}}} \frac1{A^{N/D}} \largesum_{\substack{Q' \mid \qtil\\(R+1)\text{-full part of }Q'\text{ is }A}} ~ ~ ~ \largesum_{\substack{r_1', \dots , r_M' \bmod Q'\\ \gcd(r_1', \dots , r_M', Q')=1}} 1\\
&\le \frac{(4D |\beta_M|)^{CN} \phi(\qtil)^N}{\qtil^M} \largesum_{\substack{A\mid \qtil: ~ A>1\\A\text{ is }(R+1)\text{-full}}} \frac1{A^{N/D}} \largesum_{\substack{Q' \mid \qtil\\(R+1)\text{-full part of }Q'\text{ is }A}} (Q')^M.
\end{align*}
Now any divisor $Q'$ of $\qtil$ with $(R+1)$-full part equal to $A$ must be of the form $A d$ for some $(R+1)$-free divisor $d$ of $\qtil$, and $d \le \prod_{\ell \mid \qtil} ~\ell^R \le \prod_{\ell \le C} \ell^R \le C^{CR} \ll 1$. Consequently the innermost sum in the last expression above is at most $A^M \sum_{\substack{d \mid \qtil\\d\text{ is }(R+1)\text{-free}}} d^M \ll A^M$, leading to 
\begin{align}\label{eq:S''StepTwo}\allowdisplaybreaks
|S''| \ll \frac{(4D |\beta_M|)^{CN} \phi(\qtil)^N}{\qtil^M} \largesum_{\substack{A\mid \qtil: ~ A>1\\A\text{ is }(R+1)\text{-full}}} \frac1{A^{N/D-M}},
\end{align}
Since $N \ge MD+1$, we have $N/D-M \ge 1/D$, so that for all primes $\ell$, 
$$\sum_{v \ge R+1} \frac1{\ell^{v(N/D-M)}} \le \frac1{\ell^{(R+1)(N/D-M)}} \sum_{v \ge 0} \frac1{\ell^{v/D}} \le \frac1{\ell^{(R+1)(N/D-M)}} \cdot \frac{2^{1/D}}{2^{1/D}-1} \le \frac{2D \cdot 2^{1/D}}{2^{(R+1)/D}} \le \frac{2D^2}R \le \frac12.$$
(Here, we have noted that $2^{1/D}-1 = \exp(\log 2/D)-1 \ge \log 2/D > 1/2D$ and that $2^{R/D} \ge R/D \ge 4D$.) This means that for all primes $\ell \le C$, we have
$$\log\left(1+\sum_{v \ge R+1} \frac1{\ell^{v(N/D-M)}}\right) \ll \sum_{v \ge R+1} \frac1{\ell^{v(N/D-M)}} \ll \frac1{\ell^{(R+1)(N/D-M)}} \ll \frac1{\ell^{RN/D}} \le \frac1{2^{RN/D}},$$ 
and since $\qtil$ is $C$-smooth, this leads to
\begin{equation*}
\begin{split}
\largesum_{\substack{A\mid \qtil: ~ A>1\\A\text{ is }(R+1)\text{-full}}} \frac1{A^{N/D-M}} &\le \largeprod_{\ell \mid \qtil} \left(1+\sum_{v \ge R+1} \frac1{\ell^{v(N/D-M)}}\right)-1 = \exp\left(O\left(\frac1{2^{RN/D}}\right)\right)-1 \ll \frac1{2^{RN/D}}.
\end{split}
\end{equation*}
Inserting this into \eqref{eq:S''StepTwo}, we obtain
$$|S''| \ll \left(\frac{(4D |\beta_M|)^C}{2^{R/D}}\right)^N \frac{\phi(\qtil)^N}{\qtil^M} \le C^{-N} \frac{\phi(\qtil)^N}{\qtil^M},$$
noting in the last step that ${(4D |\beta_M|)^C}\big/{2^{R/D}} \le {D (4D |\beta_M|)^C} \big/R \le C^{-1}$, by the definition of $R$. From \eqref{eq:S'FinalForm}, we now obtain $$\#\mathcal V_{N, M}\left(\qtil; (w_i)_{i=1}^M\right) = S'+S'' = \left(\frac{Q_0}{\qtil}\right)^M \phi(\qtil)^N \left\{\frac{\#\mathcal V_{N, M}\left(Q_0; (w_i)_{i=1}^M\right)}{\phi(Q_0)^N} + O\left(C^{-N}\right)\right\}.$$
Finally, writing $\#\VNMqwi = \#\mathcal V_{N, M}\left(\qtil; (w_i)_{i=1}^M\right) \prod_{\ell^e \parallel q:~\ell>C}\#\mathcal V_{N, M}\left(\ell^e; (w_i)_{i=1}^M\right)$, and invoking the estimate above for $\#\mathcal V_{N, M}\left(\qtil; (w_i)_{i=1}^M\right)$ in conjunction with \eqref{eq:Largeprime} for all the powers $\ell^e \parallel q$ of primes $\ell>C$, we obtain the estimate claimed in Proposition \ref{prop:Vqwi_ReducnToBddModulus}.
\end{proof}
\section{Joint equidistribution without input restriction: Proof of Theorem \ref{thm:UnrestrictedInput}}
By Proposition \ref{prop:convenientprop}, it remains to show that the count of inconvenient $n \le x$ for which all the $g_i(n) \equiv b_i \pmod q$ is $o(x/q^M)$ as $x \rightarrow \infty$ in the prescribed ranges of $q$. Setting $z \coloneqq x^{1/\log_2 x}$, we first remove from these $n \le x$, the ones that either have $P(n) \le z$ or have a repeated prime factor exceeding $y$. By known estimates on smooth numbers \cite[Theorem 5.13\text{ and }Corollary 5.19, Chapter III.5]{tenenbaum15}, the number of $n \le x$ having $P(n) \le z$ is $O\left(x/(\log x)^{(1+o(1)) \log_3 x}\right)$, and as seen before, the number of $n \le x$ having a repeated prime factor exceeding $y$ is $O(x/y)$. Both of these bounds being $o(x/q^M)$, it suffices to consider the contribution $\Sigma_0$ of those inconvenient $n \le x$ which have $P(n)>z$ and do not possess any repeated prime factor exceeding $y$. 

By the definition of ``inconvenient", any $n$ counted in $\Sigma_0$ must also have $P_J(n) \le y$, and hence can be written in the form $n = mP$, where $P \coloneqq P(n)>z$, $P_J(m) \le y$ and $\gcd(m, P)=1$. As such, $g_i(n) = g_i(m)+G_i(P)$, and the congruence $g_i(n) \equiv b_i \pmod q$ shows that $P$ mod $q$ lies in the set $\mathcal V_{1, M}\left(q; (b_i-g_i(m))_{i=1}^M\right)$. Setting 
$$\xiG(q) \coloneqq \max \{\#\mathcal V_{1, M}\left(q; (w_i)_{i=1}^M \right): w_1, \dots , w_M \bmod q\},$$
the Brun-Titchmarsh theorem shows that for a given $m$, the number of possibilities for $P$ is no more than
\begin{equation}\label{eq:NoOfPbound}
\largesum_{\substack{z< P \le x/m\\P \bmod q ~ \in \mathcal V_{1, M}\left(q; (b_i-g_i(m))_{i=1}^M\right)}} 1 ~ ~ \ll ~ \xiG(q) \frac{x/m}{\phi(q)\log(z/q)} \ll ~ \frac{\xiG(q)}{\phi(q)} ~ \frac{x \log_2 x}{m \log x}.
\end{equation}
To estimate the sum of $1/m$ over $m \le x$ having $P_J(m) \le y$, we write each such $m$ in the form $BA$ where $P(B) \le y < P^-(A)$ and $\Omega(A) \le J$. As such, the sum of the reciprocals of the possible $A$ is at most
$$\largesum_{\substack{A \le x\\\Omega(A) \le J}} \frac1A \le \left(1+\sum_{p \le x } \frac1p \right)^J \le (2 \log_2 x)^J \le \exp\left(O((\log_3 x)^2)\right),$$
while the sum of the reciprocals of the possible $B$ is no more than
$$\largesum_{B:~ P(B) \le y} \frac1B \le \largeprod_{p \le y} \left(1+\frac1p+O \left(\frac1{p^2}\right)\right) \le \exp\left(\sum_{p \le y} \frac1p + O(1)\right) \ll \log y.$$
Collecting estimates, we obtain 
\begin{equation}\label{eq:Sum1/mOverPJ(m)ley}
\largesum_{\substack{m \le x\\P_J(m) \le y}} \frac1m \ll (\log x)^{\delta/2} \exp\left(O((\log_3 x)^2)\right),
\end{equation}
which from the bound \eqref{eq:NoOfPbound} reveals that
\begin{equation}\label{eq:Sigma0GenBound}
\Sigma_0 \ll \frac{\xiG(q)}{\phi(q)} \frac{x \log_2 x}{(\log x)^{1-\delta/2}} \exp\left(O((\log_3 x)^2)\right) \ll \frac{\xiG(q)}q \frac x{(\log x)^{1-2\delta/3}}.
\end{equation}
We now proceed to show the assertions in the three subparts of the theorem. 

(i), (ii) If at least one of $G_1, \dots , G_M$ is linear, then $\xiG(q) \ll 1$ and we obtain $\Sigma_0 \ll x/q(\log x)^{1-2\delta/3}$. This is $o(x/q^M)$ as soon as $q^{M-1} \le (\log x)^{1-\delta}$. This condition is tautological if $M=1$, and for $M \ge 2$ it is equivalent to $q \le (\log x)^{(1-\delta)/(M-1)}$.

If $q$ is squarefree, then with $D_1 = \deg G_1$, we see that $\#\mathcal V_{1, M}\left(q; (w_i)_{i=1}^M\right) \le \#\mathcal V_{1, 1}\left(q; w_1 \right)$ $=\prod_{\ell \mid q} \#\mathcal V_{1, 1}\left(\ell; w_1 \right) \ll (D_1)^{\omega(q)} \le (\log x)^{\delta/100}$. (Here we have noted that for any sufficiently large $\ell$, the polynomial $G_1(T)-w_1$ cannot vanish identically mod $\ell$, and hence has at most $D_1$ roots mod $\ell$.) As such, from \eqref{eq:Sigma0GenBound}, it follows that $\Sigma_0 \ll x/q(\log x)^{1-3\delta/4}$. This is automatically $o(x/q^M)$ if $M=1$, while for $M \ge 2$, we need only assume that $q \le (\log x)^{(1-\delta)/(M-1)}$.

(iii) Finally, assume (by relabelling if necessary) that $\deg G_1 = \Dmin$. By a result of Konyagin \cite{konyagin79b,konyagin79a} we have $\#\mathcal V_{1, M}\left(q; (w_i)_{i=1}^M\right) \le \#\mathcal V_{1, 1}\left(q; w_1 \right) \ll q^{1-1/\Dmin}$. (To be precise, we apply Konyagin's bound to the polynomial congruence $(G_1(T) - w_1)/d \equiv 0 \pmod{q/d}$, where $d$ is the greatest common divisor of $q$ and the coefficients of the polynomial $G_1(T)-w_1$. Note that each solution mod $q/d$ lifts to a solution mod $q$ in $\le d \ll 1$ ways.) Consequently, we obtain $\Sigma_0 \ll x/q^{1/\Dmin} (\log x)^{1-2\delta/3}$. This is $o(x/q^M)$ as soon as $q^{M-1/\Dmin} \le (\log x)^{1-\delta}$, completing the proof of the theorem.
\subsection{Optimality of range of $q$ in Theorem \ref{thm:UnrestrictedInput}}\label{subsec:OptimalityUnrestricted}
We will now construct polynomials $G_1, \dots , G_M$ which will show that the various restrictions on the range of $q$ in Theorem \ref{thm:UnrestrictedInput} are all essentially optimal. To that end, let $G \in \Z[T]$ be any monic polynomial having a nonzero integer root $a$. Let $G_i(T) \coloneqq G(T)^i$, so that the polynomials $\{G_i'\}_{i=1}^M$ having distinct degrees are automatically $\Q$-linearly independent. Letting $\COG$ be the constant coming from \eqref{eq:COGDefn}, Corollary \ref{cor:JtEquidLargePD} shows that any integer $q$ having $P^-(q)>\COG$ lies in $\Qgfam$. Moreover, any prime $p$ satisfying $p \equiv a \pmod q$ also satisfies $G(p)\equiv 0 \pmod q$, hence also $g_i(p) = G_i(p) = G(p)^i \equiv 0 \pmod q$ for all $i$. As such, for all $q \le (\log x)^K$ having $P^-(q)>\max\{|a|, \COG\}$, the Siegel--Walfisz Theorem yields
\begin{equation*}
\largesum_{\substack{n \le x\\(\forall i) ~ g_i(n) \equiv 0 \pmod q}} 1 \ge
\largesum_{\substack{p \le x\\p \equiv a \pmod q}} 1 \gg \frac x{\phi(q)\log x} \gg \frac x{q\log x}.
\end{equation*}
For any $M \ge 2$, this last expression grows strictly faster than $x/q^M$ as soon as $q^{M-1}$ grows faster than $\log x$, for instance if $q > (\log x)^{(1+\delta)/(M-1)}$. This construction shows that the range of $q$ in Theorem \ref{thm:UnrestrictedInput}(ii) is essentially optimal.

Now consider any $M \ge 1$, $D \ge 1$, and let $G(T) \coloneqq (T-1)^d$. Then with $G_i(T) = G(T)^i$, we see that $\Dmin = d$. For moduli $q$ of the form $q_1^d$ (for some $q_1>1$), any prime $p \equiv 1 \pmod{q_1}$ satisfies $G(p)= (p-1)^d \equiv 0 \pmod q$. Hence, if $q_1 \le (\log x)^K$ has $P^-(q_1)>\COG$, then $q = q_1^d \le (\log x)^{Kd}$ also has $P^-(q)>\COG$, and we find that on the one hand $q \in \Qgfam$, while on the other,
\begin{equation*}
\largesum_{\substack{n \le x\\(\forall i) ~ g_i(n) \equiv 0 \pmod q}} 1 \ge
\largesum_{\substack{p \le x\\p \equiv 1 \pmod {q_1}}} 1 \gg \frac x{\phi(q_1)\log x} \gg \frac x{q^{1/d}\log x}.
\end{equation*}
This last expression grows strictly faster than $x/q^M$ as soon as $q^{M-1/d}$ grows faster than $\log x$, for instance if $q > (\log x)^{(1+\delta)(M-1/d)^{-1}}$. Since $d = \Dmin$, this example shows that the range of $q$ in Theorem \ref{thm:UnrestrictedInput}(iii) is essentially optimal as well. 
\section{Complete uniformity for general moduli: Proof of Theorem \ref{thm:restrictediput_generalq}}\label{sec:CompleteUnifGenq}
In section \ref{sec:ConvnMainTerm}, we had defined $J = \lfloor \log_3 x \rfloor$ and for the purposes of this theorem, we took $\delta \coloneqq 1$, so that $y = \exp((\log x)^{1/2})$. If $x$ is sufficiently large then any convenient $n$ has $P_{MD+1}(n) \ge P_J(n) \ge y > q$. Moreover, by \cite[Lemma 2.3]{PSR22} the number of $n \le x$ having $P_{MD+1}(n) \le q$ is $o(x)$. By Proposition \ref{prop:convenientprop}, it remains to show that there are $o(x/q^M)$ many inconvenient $n \le x$ having $P_{MD+1}(n) > q$ and satisfying $g_i(n) \equiv b_i \pmod q$ for all $i$. 

Now by the arguments in the beginning of the previous section, the number of $n \le x$ which either have $P(n) \le z = x^{1/\log_2 x}$ or have a repeated prime factor exceeding $y$ is $o(x/q^M)$. As such, in order to complete the proof of the theorem, it suffices to show that
\begin{equation}\label{eq:GenqUnif_Remaining}
\largesum_{\substack{n \le x: ~ P_{MD+1}(n)>q\\P_J(n) \le y; ~ P(n)>z\\p>y \implies p^2 \nmid n\\(\forall i) ~ g_i(n) \equiv b_i \pmod q}} 1 ~ ~ ~\ll ~ \frac x{q^M(\log x)^{1/3}}
\end{equation}
uniformly in $q \le (\log x)^K$ and in residues $(b_1, \dots , b_M)$ mod $q$. 

Assume first that $M \ge 2$. To show \eqref{eq:GenqUnif_Remaining} write the count on the left hand side as 
$$\Sigma_0 + \Sigma_1 + \Sigma_2 + \Sigma,$$
where
\begin{itemize}
\item $\Sigma_0$ counts those $n$ which are exactly divisible by at least $MD+1$ many distinct primes exceeding $q$,
\item For $r \in \{1, 2\}$, $\Sigma_r$ counts the $n$ that are exactly divisible by at least $(M-r)D+1$ but at most $(M-r+1)D$ many distinct primes exceeding $q$, and
\item $\Sigma$ counts the remaining $n$, namely, those that are exactly divisible by at most $(M-2)D$ many distinct primes exceeding $q$. 
\end{itemize}
We proceed to show that the expression on the right hand side of \eqref{eq:GenqUnif_Remaining} bounds each of $\Sigma_0$, $\Sigma_1$, $\Sigma_2$ and $\Sigma$. To do this, we shall bound the cardinalities of the sets $\VNMqwi$ that arise by discarding some of the congruences defining the set. The following consequence of Proposition \ref{prop:Vqwi_ReducnToBddModulus} will be useful: for any fixed $r \in \{0, 1, \dots , M-1\}$, we have 
\begin{equation}\label{eq:ConvenientPropConseq}
\#\mathcal V_{(M-r)D+1, M-r}\left(q; (w_i)_{i=1}^{M-r}\right) \ll \frac{\phi(q)^{(M-r)D+1}}{q^{M-r}} \exp\left(O\left((\log q)^{1-1/D}\right)\right)  
\end{equation}
uniformly in moduli $q>1$ and in residue classes $(w_1, \dots , w_M)$ mod $q$. Here, we have noted that $\{G_i'\}_{i=1}^{M-r}$ are $\Q$-linearly independent, that $\max_{1 \le i \le M-r} \deg G_i \le D$, and that
\begin{equation*}
\begin{split}
\largeprod_{\ell \mid q} \left(1+O\left(\frac1{\ell^{1/D}}\right)\right) \le \exp\left(O\left(\largesum_{\ell \le \omega(q)}\frac1{\ell^{1/D}}\right)\right) \ll \exp\left(O\left((\log q)^{1-1/D}\right)\right),
\end{split}    
\end{equation*}
with the last sum on $\ell$ being bounded by partial summation and Chebyshev's estimates.

\textit{Bounding $\Sigma_0$:} Any $n$ counted in $\Sigma_0$ is exactly divisible by at least $\MDOne$ many prime factors exceeding $q$ and has $P(n)>z$, $P_J(n) \le y$. Hence, $n$ can be written in the form $mP_1 \cdots P_{\MDOne}$, where $P_1 \coloneqq P(n) > z$, $q < P_{\MDOne} < \cdots < P_1$, $P_J(m) \le y$ and $\gcd(m, P_1 \cdots P_{\MDOne}) = 1$. As such, $g_i(n) = g_i(m) + \sum_{1 \le j \le \MDOne} G_i(P_j)$ and the congruences $g_i(n) \equiv b_i \pmod q$ force $(P_1, \dots , P_{\MDOne})$ mod $q$ to lie in the set $V_m \coloneqq \mathcal V_{\MDOne, M}\left(q; (b_i-g_i(m))_{i=1}^M\right)$.

Given $m$ and $\widehat{v} \coloneqq (v_1, \dots , v_{\MDOne}) \in V_m$, we count the number of possible $P_1, \dots , P_{\MDOne}$ satisfying $(P_1, \dots , P_{\MDOne}) \equiv \widehat{v}$ mod $q$. For a given choice of $P_2, \dots , P_{\MDOne}$, the number of possible $P_1$ is, by the Brun-Titchmarsh inequality, no more than
$$\largesum_{\substack{z<P_1 \le x/mP_2 \cdots P_{\MDOne}\\P_1 \equiv v_1 \pmod q}} 1 \ll \frac{x/mP_2 \cdots P_{\MDOne}}{\phi(q) \log(z/q)} \ll \frac{x \log_2 x}{\phi(q) mP_2 \cdots P_{\MDOne} \log x}.$$
For each $j \in \{2, \dots , \MDOne\}$, the sum on $P_j$ is, by Brun-Titchmarsh and partial summation, no more than 
$$\largesum_{\substack{q<p\le x\\p\equiv v_j \pmod q}} \frac1p \ll \frac {\log_2 x}{\phi(q)}.$$ 
Hence, given $m$ and $\widehat{v} = (v_1, \dots , v_{\MDOne}) \in V_m$, the number of possible $P_1, \dots , P_{\MDOne}$ satisfying $(P_1, \dots , P_{\MDOne}) \equiv \widehat{v}$ mod $q$ is 
$$\ll \frac{x (\log_2 x)^{O(1)}}{\phi(q)^{\MDOne} m \log x},$$
leading to 
$$\Sigma_0 \ll \frac{x (\log_2 x)^{O(1)}}{\log x} \largesum_{\substack{m \le x\\P_J(m) \le y}} \frac1m \cdot \frac{\#V_m}{\phi(q)^{\MDOne}}.$$
Using \eqref{eq:ConvenientPropConseq} to bound $V_m = \mathcal V_{\MDOne, M}\left(q; (b_i-g_i(m))_{i=1}^M\right)$, followed by \eqref{eq:Sum1/mOverPJ(m)ley} to bound the resulting sum on $m$, we deduce that
$$\Sigma_0 \ll \frac{x(\log_2 x)^{O(1)}}{q^M \log x}  \exp\left(O\left((\log q)^{1-1/D}\right)\right) \largesum_{\substack{m \le x\\P_J(m) \le y}} \frac1m \ll \frac x{q^M (\log x)^{1/3}},$$
yielding the desired bound for $\Sigma_0$. It is to be noted that this bound on $\Sigma_0$ holds true for any $M \ge 1$.

\textit{Bounding $\Sigma_1$:} Recall that $\OmStqn \coloneqq \largesum_{\substack{p^k \parallel n\\p>q, ~ k>1}} k$ counts (with multiplicity) the number of prime factors of $n$ exceeding $q$ that appear to an exponent larger than $1$ in the prime factorization of $n$; as such, the squarefull part of $n$ (i.e., the largest squarefull divisor of $n$) exceeds $q^{\OmStqn}$. 

Now, any $n$ counted in $\Sigma_1$ is exactly divisible by least $(M-1)D+1$ but at most $MD$ many distinct primes exceeding $q$. Since $P_{\MDOne}(n)>q$, it follows that $\OmStqn \ge 2$, so that the squarefull part of $n$ exceeds $q^2$. As such, $n$ can be written in the form $mSP_{(M-1)D+1} \cdots P_1$, where $m, S, P_{(M-1)D+1}, \dots , P_1$ are pairwise coprime, $P_1 \coloneqq P(n)>z$, $q< P_{(M-1)D+1} < \cdots < P_1$, $P_J(m) \le y$, and $S>q^2$ is squarefull. Since $g_i(n) = g_i(mS) + \sum_{1 \le j \le (M-1)D+1} G_i(P_j)$, the congruence conditions $g_i(n) \equiv b_i \pmod q$, considered for $1 \le i \le M-1$, force $(P_1, \dots , P_{(M-1)D+1}) \equiv \widehat{v}$ mod $q$ for some $\widehat{v} \coloneqq (v_1, \dots , v_{(M-1)D+1}) \in \mathcal V_{(M-1)D+1, M-1}\left(q; (b_i-g_i(mS))_{i=1}^{M-1}\right)$. 

Given $m, S$ and $\widehat{v}$, the argument given for bounding $\Sigma_0$ above shows that the number of possible $P_1, \dots , P_{(M-1)D+1}$ satisfying $(P_1, \dots , P_{(M-1)D+1}) \equiv \widehat{v}$ mod $q$ is 
$$\ll \frac{x(\log_2 x)^{O(1)}}{\phi(q)^{(M-1)D+1} mS\log x}.$$
This yields
$$\Sigma_1 \ll \frac{x(\log_2 x)^{O(1)}}{\log x} \largesum_{\substack{m \le x\\P_J(m) \le y}} \frac1m \sum_{S>q^2\text{ squarefull }} \frac1S \cdot \frac{\#\mathcal V_{(M-1)D+1, M-1}\left(q; (b_i-g_i(mS))_{i=1}^{M-1}\right)}{\phi(q)^{(M-1)D+1}},$$
so that by \eqref{eq:ConvenientPropConseq}, 
$$\Sigma_1 \ll \frac{x(\log_2 x)^{O(1)}}{q^{M-1} \log x} \exp\left(O\left((\log q)^{1-1/D}\right)\right) \largesum_{\substack{m \le x\\P_J(m) \le y}} \frac1m \sum_{S>q^2\text{ squarefull }} \frac1S.$$
Using \eqref{eq:Sum1/mOverPJ(m)ley} along with the bound $\sum_{S>q^{2}\text{ squarefull }} 1/S \ll 1/q$, we obtain 
$$\Sigma_1 \ll \frac{x(\log_2 x)^{O(1)}}{q^M (\log x)^{1/2}}  \exp\left(O\left((\log q)^{1-1/D} + (\log_3 x)^2\right)\right)\ll \frac x{q^M (\log x)^{1/3}},$$
showing the desired bound for $\Sigma_1$.

\textit{Bounding $\Sigma_2$:} Any $n$ counted in $\Sigma_2$ is exactly divisible by least $(M-2)D+1$ but at most $(M-1)D$ many distinct primes exceeding $q$. Since $P_{\MDOne}(n)>q$, it follows that $\OmStqn \ge MD+1 -(M-1)D = D+1$. Now assume that $D \ge 3$, so that $\OmStqn \ge 4$, and the squarefull part of $n$ exceeds $q^4$. In this case, any $n$ counted in $\Sigma_2$ can be written in the form $mSP_{(M-2)D+1} \cdots P_1$, where $m, S, P_{(M-2)D+1}, \dots , P_1$ are pairwise coprime, $P_1 \coloneqq P(n)>z$, $q< P_{(M-2)D+1} < \cdots < P_1$, $P_J(m) \le y$, and $S>q^4$ is squarefull. Since $g_i(n) = g_i(mS) + \sum_{1 \le j \le (M-2)D+1} G_i(P_j)$, the congruence conditions $g_i(n) \equiv b_i \pmod q$, considered for $1 \le i \le M-2$, force $(P_1, \dots , P_{(M-2)D+1}) \equiv \widehat{v}$ mod $q$ for some $\widehat{v} \coloneqq (v_1, \dots , v_{(M-2)D+1}) \in \mathcal V_{(M-2)D+1, M-2}\left(q; (b_i-g_i(mS))_{i=1}^{M-2}\right)$. Replicating the argument given for $\Sigma_1$ shows that 
\begin{align*}\allowdisplaybreaks
\Sigma_2 &\ll \frac{x(\log_2 x)^{O(1)}}{\log x} \largesum_{\substack{m \le x\\P_J(m) \le y}} \frac1m \sum_{S>q^{4}\text{ squarefull }} \frac1S \cdot \frac{\#\mathcal V_{(M-2)D+1, M-2}\left(q; (b_i-g_i(mS))_{i=1}^{M-2}\right)}{\phi(q)^{(M-2)D+1}}\\
&\ll \frac{x(\log_2 x)^{O(1)}}{q^{M-2} \log x} \exp\left(O\left((\log q)^{1-1/D}\right)\right) \largesum_{\substack{m \le x\\P_J(m) \le y}} \frac1m \sum_{S>q^{4}\text{ squarefull }} \frac1S\\
&\ll \frac{x(\log_2 x)^{O(1)}}{q^M (\log x)^{1/2}}  \exp\left(O\left((\log q)^{1-1/D} + (\log_3 x)^2\right)\right)\ll \frac x{q^M (\log x)^{1/3}}.
\end{align*}
showing the desired bound for $\Sigma_2$ in the case $D \ge 3$. 

Now assume that $D=2$, so that $2 \le M \le D = 2$ forces $M=2$. Any $n$ counted in $\Sigma_2$ has $P_5(n)>q$ but at most $(M-1)D=2$ of these exactly divide $n$. Hence, $n$ is either divisible by the cube of a prime exceeding $q$ or is (exactly) divisible by the squares of two distinct primes exceeding $q$. Any $n$ of the first kind can be written in the form $m p^s P$ for some primes $p, P$ satisfying $P=P(n)>z$ and $q<p<P$, and some positive integers $s, m$ satisfying $s \ge 3$, $P_J(m) \le y$. Given $m, p$ and $s$, the number of possible $P \in (z, x/mp^s]$ is $O(x/mp^s \log z)$. Summing this over all $s \ge 3$, all $p>q$, and then over all possible $m$, and invoking \eqref{eq:Sum1/mOverPJ(m)ley} in conjunction with the fact that $\sum_{p>q} 1/p^3 \ll 1/q^2$, we find that the total contribution of all $n$ of the first kind is $\ll x/q^2(\log x)^{1/3}$ which is absorbed in the desired expression. 

On the other hand, if $n$ is divisible by the squares of two distinct primes exceeding $q$, then it is of the form $m p_1^{s_1} p_2^{s_2} P$ for some primes $P, p_1, p_2$ satisfying $P=P(n)>z$ and $q<p_2<p_1<P$, and for some positive integers $m, s_1, s_2$ satisfying $s_1 \ge 2$, $s_2 \ge 2$ and $P_J(m) \le y$. Given $m, p_1, p_2, s_1, s_2$, the number of possible $P \in (z, x/m p_1^{s_1} p_2^{s_2}]$ is $O(x/m p_1^{s_1} p_2^{s_2} \log z)$. Summing this over all possible $s_i$, $p_i$, and $m$ via \eqref{eq:Sum1/mOverPJ(m)ley} and the fact that $\sum_{p>q} 1/p^2 \ll 1/q$, we deduce that the total contribution of all $n$ that are divisible by the squares of two primes is $\ll x/q^2(\log x)^{1/3}$. This establishes the desired bound on the sum $\Sigma_2$ in the remaining case $D=2$.

\textit{Bounding $\Sigma$:} Any $n$ counted in $\Sigma$ has $P_{\MDOne}(n)>q$, but no more than $(M-2)D$ of these exactly divide $n$. Since $D = \max_{1 \le i \le M} \deg G_i \ge M$, it follows that any such $n$ has $\OmStqn \ge \MDOne - (M-2)D = 2D + 1 \ge 2M+1$, so that the squarefull part of $n$ exceeds $q^{2M+1}$. Consequently, any $n$ counted in $\Sigma$ can be written in the form $mSP$, where $P \coloneqq P(n)>z$, $S>q^{2M+1}$ is squarefull and $P_J(m) \le y$. Given $m$ and $S$, the number of possible $P \in (z, x/mS]$ is $O(x/mS\log z)$. Summing this over all squarefull $S>q^{2M+1}$ and then over all $m$ by means of \eqref{eq:Sum1/mOverPJ(m)ley}, we find that 
$$\Sigma \ll \frac{x\log_2 x}{\log x} \largesum_{\substack{m \le x\\P_J(m) \le y}} \frac1m \largesum_{\substack{S>q^{2M+1}\\S\text{ squarefull}}} \frac1S \ll \frac x{q^{M+1/2} (\log x)^{1/3}},$$
yielding the desired bound for $\Sigma$, and completing the proof of the estimate \eqref{eq:GenqUnif_Remaining}, for $M \ge 2$.

The case $M=1$ is much simpler: we need only split the count in the left hand side of \eqref{eq:GenqUnif_Remaining} as $\Sigma_0+\Sigma$ where $\Sigma_0$ counts those $n$ that have no repeated prime factor exceeding $q$. As such, any $n$ counted in $\Sigma_0$ is exactly divisible by at least $D+1$ primes exceeding $q$, whereupon the exact same arguments given for the ``$\Sigma_0$" defined in the case $M \ge 2$ show that $\Sigma_0 \ll x/q(\log x)^{1/3}$. On the other hand, any $n$ counted in $\Sigma$ has a repeated prime factor exceeding $q$, and thus is of the form $mSP$, with $P \coloneqq P(n)>z$, $S>q^2$ squarefull and $P_J(m) \le y$. Proceeding as for the ``$\Sigma$" considered in the case $M \ge 2$, we obtain $\Sigma \ll x/q(\log x)^{1/3}$. This shows the estimate \eqref{eq:GenqUnif_Remaining} in the remaining case $M=1$, completing the proof of theorem. \hfill \qedsymbol
\section{Complete uniformity in squarefree moduli: Proof of Theorem \ref{thm:restrictediput_squarefreeq}}
Arguing as in the beginning of the previous section, in order to complete the proof of the theorem, it suffices to show the following analogue of \eqref{eq:GenqUnif_Remaining}
\begin{equation}\label{eq:SqfreeqUnif_Remaining}
\largesum_{\substack{n \le x: ~ P_{2M}(n)>q\\P_J(n) \le y; ~ P(n)>z\\p>y \implies p^2 \nmid n\\(\forall i) ~ g_i(n) \equiv b_i \pmod q}} 1 ~ ~ ~\ll ~ \frac x{q^M(\log x)^{1/3}}
\end{equation}
uniformly in squarefree $q \le (\log x)^K$ and in residues $(b_1, \dots , b_M)$ mod $q$. 

The following analogue of \eqref{eq:ConvenientPropConseq} will be useful for this purpose: for each $r \in \{0, 1, \dots , M-1\}$, we have 
\begin{equation}\label{eq:V2MMSqfreeq}
\#\mathcal V_{2(M-r), M-r}\left(q; (w_i)_{i=1}^{M-r}\right) \le \lambda^{\omega(q)} \frac{\phi(q)^{2(M-r)}}{q^{M-r}}  
\end{equation}
uniformly for squarefree $q>1$ and in residue classes $(w_1, \dots , w_{M-r})$ mod $q$, for some constant $\lambda \coloneqq \lambda(\widehat{G})>1$. It suffices to show this bound for $r=0$ for then it may be applied with $M-r$ playing the role of $M$ (recalling that $\{G_i'\}_{i=1}^{M-r}$ are $\Q$-linearly independent for any such $r$). 

As in Proposition \ref{prop:Vqwi_ReducnToBddModulus}, we let $C \coloneqq \CG$ be a constant exceeding $\max\{\COG, (2D)^{2D+4}\}$, with $\COG$ defined in \eqref{eq:COGDefn}. Then for all $\ell \le \CG$, we have trivially
\begin{equation}\label{eq:V2MMSmallPrimes}
\#\mathcal V_{2M, M}\left(\ell; (w_i)_{i=1}^{M}\right) \le \phi(\ell)^{2M} \le \lambda_1 \frac{\phi(\ell)^{2M}}{\ell^M}
\end{equation} 
by fixing $\lambda_1 \coloneqq \lambda_1(\widehat{G}) > \CG^M$. 

Now consider a prime $\ell>\CG$. By orthogonality we can write, as in \eqref{eq:primepowerOrth}, 
$$\#\mathcal V_{2M, M}\left(\ell; (w_i)_{i=1}^M\right) = \frac{\phi(\ell)^{2M}}{\ell^M} \left\{1+\frac1{\phi(\ell)^{2M}} \largesum_{(r_1, \dots , r_M) \not\equiv (0, \dots , 0) \bmod \ell} e\left(-\frac1\ell\sum_{i=1}^M r_i w_i\right) \left(Z_{\ell; \, r_1, \dots, r_M}\right)^{2M}\right\},$$
where $Z_{\ell; \, r_1, \dots, r_M} \coloneqq \largesum_{v \bmod \ell} \chi_{0, \ell}(v) e\left(\frac1{\ell}\largesum_{i=1}^M r_i G_i(v)\right)$. Since $\ell>\CG>\COG$, the polynomials $\{G_i'\}_{i=1}^M$ must be $\F_\ell$-linearly independent, so that for each $(r_1, \dots , r_M) \not\equiv (0, \dots , 0) \bmod \ell$, the polynomial $\sum_{i=1}^M r_i G_i(T)$ does not reduce to a constant mod $\ell$. As such, the Weil bound (Proposition \ref{prop:Weil}) yields $|Z_{\ell; \, r_1, \dots, r_M}| \le D\ell^{1/2}$, leading to 
\begin{equation}\label{eq:V2MMLargePrimes} 
\#\mathcal V_{2M, M}\left(\ell; (w_i)_{i=1}^M\right) = \frac{\phi(\ell)^{2M}}{\ell^M} \left\{1+O\left(\ell^M \frac{(D\ell^{1/2})^{2M}}{\phi(\ell)^{2M}}\right)\right\} \le \lambda_2 \frac{\phi(\ell)^{2M}}{\ell^M},
\end{equation}
for some constant $\lambda_2 \coloneqq \lambda_2(\widehat{G}) > \CG^M$. Finally, we choose $\lambda \coloneqq \max\{\lambda_1, \lambda_2\}$ and write, for any squarefree $q>1$, $\#\mathcal V_{2M, M}\left(q; (w_i)_{i=1}^M\right) = \largeprod_{\substack{\ell \mid q: ~\ell \le C}} \#\mathcal V_{2M, M}\left(\ell; (w_i)_{i=1}^M\right) \cdot \largeprod_{\substack{\ell \mid q:~ \ell > C}} \#\mathcal V_{2M, M}\left(\ell; (w_i)_{i=1}^M\right)$. Combining \eqref{eq:V2MMSmallPrimes} for all the prime divisors $\ell \le C$ with \eqref{eq:V2MMLargePrimes} for all the prime divisors $\ell > C$, we obtain the desired bound \eqref{eq:V2MMSqfreeq} for $r = 0$. As argued before, this also implies \eqref{eq:V2MMSqfreeq} for any $r \in \{0, 1, \dots , M-1\}$. 

Coming to the proof of \eqref{eq:SqfreeqUnif_Remaining}, we write the count on the left hand side as 
$$\Sigma_1 + \Sigma_2 + \dots + \Sigma_M + \Sigma,$$
where 
\begin{itemize}
\item $\Sigma_1$ counts those $n$ which are exactly divisible by at least $2M$ many distinct primes exceeding $q$,
\item For each $r \in \{1, \dots , M-1\}$, $\Sigma_{r+1}$ counts the $n$ that are exactly divisible by either $2M-2r$ many or by $2M-2r+1$ many distinct primes exceeding $q$, and
\item $\Sigma$ counts the remaining $n$, namely, those that are exactly divisible by at most one prime exceeding $q$. 
\end{itemize}
\textit{Bounding $\Sigma_1$:} Any $n$ counted in $\Sigma_1$ can be written in the form $mP_{2M} \cdots P_1$, where $P_1 \coloneqq P(n)>z$, $q<P_{2M} < \cdots < P_1$, $P_J(m) \le y$ and $\gcd(m, P_{2M} \cdots P_1)=1$. As such, the congruences $g_i(n) \equiv b_i \pmod q$ force $(P_1, \dots , P_{2M}) \equiv \widehat{v}$ mod $q$ for some $\widehat{v} \coloneqq (v_1, \dots , v_{2M}) \in \mathcal V_{2M, M}\left(q; (b_i-g_i(m))_{i=1}^M\right)$. Given $m$ and $\widehat{v}$, the arguments in the previous section show that the number of possible $P_1, \dots , P_{2M}$ satisfying $(P_1, \dots , P_{2M}) \equiv \widehat{v}$ mod $q$ is 
$$\ll \frac{x(\log_2 x)^{O(1)}}{\phi(q)^{2M} m\log x}.$$
Consequently,
$$\Sigma_1 \ll \frac{x(\log_2 x)^{O(1)}}{\log x} \largesum_{\substack{m \le x\\P_J(m) \le y}} \frac1m \cdot \frac{\#\mathcal V_{2M, M}\left(q; (b_i-g_i(m))_{i=1}^M\right)}{\phi(q)^{2M}}.$$
Using \eqref{eq:V2MMSqfreeq} to bound the cardinality $\#\mathcal V_{2M, M}\left(q; (b_i-g_i(m))_{i=1}^M\right)$ in conjunction with \eqref{eq:Sum1/mOverPJ(m)ley} to bound the resulting sum on $m$, we obtain
$$\Sigma_1 \ll \lambda^{\omega(q)}\frac{x(\log_2 x)^{O(1)}}{q^M (\log x)^{1/2}}  \exp\left(O\left((\log_3 x)^2\right)\right) \ll \frac x{q^M (\log x)^{1/3}},$$
showing the desired bound for $\Sigma_1$. 

\textit{Bounding $\Sigma_2, \dots , \Sigma_M$:} We start by making the following general observation: let $E$ be a set of primes and for a positive integer $N$, let $\OmStEN \coloneqq \largesum_{\substack{p^k \parallel n\\p \in E, ~ k>1}} k$ denote the number of prime divisors of $N$ (counted with multiplicity) lying in the set $E$ and appearing to an exponent greater than $1$ in the prime factorization of $N$. Then for any $t \ge 2$, any positive integer $N$ having $\OmStEN \ge t$ is divisible by $p_1^{\alpha_1} \cdots p_s^{\alpha_s}$ for some distinct primes $p_1, \dots, p_s \in E$, and integers $\alpha_1, \dots, \alpha_s \ge 2$ summing to $t$ or $t+1$. More precisely, there exist positive integers $s$, $m$, $\alpha_1, \dots , \alpha_s, \beta_1, \dots , \beta_s$ and distinct primes $p_1, \dots , p_s \in E$ such that $\alpha_1, \dots , \alpha_s \ge 2$, $\sum_{i=1}^s \alpha_i \in \{t, t+1\}$, $\gcd(m, p_1 \cdots p_s) = 1$, $N = m p_1^{\beta_1} \cdots p_s^{\beta_s}$ and $\beta_i \ge \alpha_i$ for all $i \in [s]$.

This is seen by a simple induction on $t$, the case $t = 2$ being clear with $(\alpha_1, \dots , \alpha_s) = (2)$ and the case $t=3$ being clear with $(\alpha_1, \dots , \alpha_s) \in \{(3), (2, 2)\}$. Consider any $T \ge 4$, assume that the result holds for all $t<T$, and let $N$ be a positive integer with $\OmStEN \ge T$. Let $p_1$ be the largest prime divisor of $N$ lying in the set $E$ and satisfying $p_1^2\mid n$, and let $\beta_1 \coloneqq v_{p_1}(N) \ge 2$. If $\beta_1 \ge T-1$, then we are done with $(\alpha_1, \dots , \alpha_s)$ being $(T)$ or $(T-1, 2)$, so suppose $\beta_1 \le T-2$. Then the positive integer $N' \coloneqq N/p_1^{\beta_1}$ is not divisible by $p_1$, and has $\Omega^*_E(N') \ge T-\beta_1 \ge T-(T-2) = 2$. As such, by the inductive hypothesis applied to $N'$ and $t \coloneqq T-\beta_1$, there exist $s, m, \alpha_2, \dots , \alpha_s, \beta_2, \dots , \beta_s$ and distinct primes $p_2, \dots , p_s \in E$ satisfying $\alpha_2, \dots , \alpha_s \ge 2$, $\sum_{i=2}^s \alpha_i \in \{T-\beta_1, T-\beta_1+1\}$, $\gcd(m, p_2 \cdots p_s) = 1$, $N' = m p_2^{\beta_2} \cdots p_s^{\beta_s}$ and $\beta_i \ge \alpha_i$ for all $i \in \{2, \dots , s\}$. Since $p_1 \nmid N'$, we see that the primes $p_1, \dots , p_s \in E$ must all be distinct and that $\gcd(m, p_1 \cdots p_s)=1$. Consequently, with $\alpha_1 \coloneqq \beta_1 \ge 2$, we have $N = p_1^{\beta_1} N' = m p_1^{\beta_1} p_2^{\beta_2} \cdots p_s^{\beta_s}$ with $\sum_{i=1}^s \alpha_i \in \{T, T+1\}$ and with $\beta_i \ge \alpha_i$ for all $i \in [s]$. This completes the induction step, establishing the claimed observation.

With this observation in hand, we note that for each $r \in \{1, \dots , M-1\}$, any $n$ counted in the sum $\Sigma_{r+1}$ is of the form $m p_1^{\beta} \cdots p_s^{\beta_s} P_{2M-2r} \cdots P_1$ where all of the following hold:
\begin{enumerate}
\item[(i)] $P_1 \coloneqq P(n)>z$; 
\item[(ii)] $q<P_{2M-2r} < \cdots < P_1$;
\item[(iii)] $p_1, \dots , p_s>q$;
\item[(iv)] $\beta_1 \ge \alpha_1, \dots , \beta_s \ge \alpha_s$ for some positive integers $\alpha_1, \dots , \alpha_s$ at least $2$ summing to either $\max\{2, 2r-1\}$ or to $2r$;
\item[(v)] $P_J(m) \le y$; 
\item[(vi)]$m, p_1, \dots , p_s,$ $P_{2M-2r}, \dots , P_1$ are all pairwise coprime.
\end{enumerate}
Indeed, any $n$ counted in $\Sigma_{r+1}$ is exactly divisible by at least $2M-2r$ but at most $2M-2r+1$ many primes (counted with multiplicity) exceeding $q$. Hence in the case $r=1$ we have $\OmStqn \ge 2$ while for $r \in \{2, \dots , M-1\}$, we have $\OmStqn \ge 2M-(2M-2r+1) \ge 2r-1$, so altogether $\OmStqn \ge \max\{2, 2r-1\}$. Let $P_1, P_2, \dots , P_{2M-2r}$ be primes exceeding $q$ that exactly divide $n$, and satisfy $P_1 \coloneqq P(n)>z$ and $P_{2M-2r} < \cdots < P_2<P_1$. Then with $n' \coloneqq n/P_1 \cdots P_{2M-2r}$, we still have $\Omega^*_{>q}(n') = \OmStqn \ge \max\{2, 2r-1\}$ and $\gcd(n', P_1 \cdots P_{2M-2r})=1$. Invoking the above observation for $N \coloneqq n'$, $t \coloneqq \max\{2, 2r-1\}$ and $E$ the set of primes exceeding $q$, we find that $n' = m p_1^{\beta} \cdots p_s^{\beta_s}$ for some $s \ge 1$, primes $p_1, \dots , p_s>q$ and positive integers $m, \beta_1, \dots , \beta_s$ such that $m, p_1, \dots , p_s$ are pairwise coprime, and $\beta_1 \ge \alpha_1, \dots , \beta_s \ge \alpha_s$ for some positive integers $\alpha_1, \dots , \alpha_s$ at least $2$ summing to either $\max\{2, 2r-1\}$ or $2r$. (Here, we have recalled that in the case $t=2$, the tuple $(\alpha_1, \dots , \alpha_s) = (2)$ was sufficient.)  Altogether, we find that $n = n' P_1 \cdots P_{2M-2r} = m p_1^{\beta} \cdots p_s^{\beta_s} P_1 \cdots P_{2M-2r}$, with $m, p_1, \dots , p_s, \beta_1, \dots , \beta_s, P_1, \dots , P_{2M-2r}$ satisfying the conditions (i)-(vi). 

Consequently, $g_i(n) = g_i(m p_1^{\beta_1} \cdots p_s^{\beta_s}) + \sum_{j=1}^{2M-2r} G_i(P_j)$, and the conditions $g_i(n) \equiv b_i \pmod q$ for $i \in [M-r]$ force $(P_1, \dots , P_{2M-2r}) \equiv \widehat{v}$ mod $q$ for some element $\widehat{v} \coloneqq (v_1, \dots , v_{2M-2r})$ of the set $\mathcal V_{2M-2r, M-r}\left(q; (b_i-g_i(m p_1^{\beta} \cdots p_s^{\beta_s}))_{i=1}^{M-r}\right)$. Given $m, s, \alpha_1, \dots , \alpha_s$, $p_1, \dots , p_s$, $\beta_1, \dots , \beta_s$ and $\widehat{v}$, the arguments in the previous section show that the number of possible $P_1, \dots ,$ $P_{2M-2r}$ satisfying $(P_1, \dots , P_{2M-2r}) \equiv \widehat{v}$ mod $q$ is 
$$\ll \frac{x(\log_2 x)^{O(1)}}{\phi(q)^{2M-2r} mp_1^{\beta_1} \cdots p_s^{\beta_s} \log x}.$$
Using \eqref{eq:V2MMSqfreeq} to bound the cardinality of the set $\mathcal V_{2M-2r, M-r}\left(q; (b_i-g_i(m p_1^{\beta} \cdots p_s^{\beta_s}))_{i=1}^{M-r}\right)$, we find that 
\begin{equation*}
\Sigma_{r+1} \ll \lambda^{\omega(q)} \frac{x (\log_2 x)^{O(1)}}{q^{M-r} \log x} \largesum_{\substack{m \le x\\P_J(m) \le y}} \frac1m \largesum_{\substack{s \ge 1; ~ \alpha_1, \dots , \alpha_s \ge 2\\\alpha_1 + \dots +\alpha_s \in \{2r-1, 2r\}}} ~ ~ \largesum_{\substack{p_1, \dots , p_s > q\\\beta_1 \ge \alpha_1, \dots , \beta_s \ge \alpha_s}} \frac1{p_1^{\beta_1} \cdots p_s^{\beta_s}}.
\end{equation*}
Now, the sum on $p_1, \dots , p_s, \beta_1, \dots , \beta_s$ is no more than
$$\largeprod_{i=1}^s \left(\largesum_{p_i>q} ~ \largesum_{\beta_i \ge \alpha_i} \frac1{p_i^{\beta_i}}\right) \ll \largeprod_{i=1}^s \left(\largesum_{p_i>q} \frac1{p_i^{\alpha_i}}\right) \ll \frac1{q^{\alpha_1+\cdots+\alpha_s-s}}.$$
In addition since $s \ge 1$ and $\sum_{i=1}^s \alpha_i \ge 2r-1$ and each $\alpha_i \ge 2$, we find that $\sum_{i=1}^s \alpha_i - s \ge r$: indeed, from the bound $\sum_{i=1}^s \alpha_i - s \ge 2s-s = s \ge 1$, it remains to only see that for $r \ge 2$, we have $\sum_{i=1}^s \alpha_i - s \ge \max\{s, 2r-1-s\} \ge r$. Collecting estimates, we obtain
$$\Sigma_{r+1} \ll \lambda^{\omega(q)} \frac{x (\log_2 x)^{O(1)}}{q^M \log x} \largesum_{\substack{m \le x\\P_J(m) \le y}} \frac1m \largesum_{\substack{s \ge 1; ~ \alpha_1, \dots , \alpha_s \ge 2\\\alpha_1 + \dots +\alpha_s \in \{2r-1, 2r\}}} 1.$$
But since there are $O(1)$ many possible $s \ge 1$ and tuples $(\alpha_1, \dots , \alpha_s)$ of positive integers summing to $2r-1$ or to $2r$, this automatically leads to 
$$\Sigma_{r+1} \ll \lambda^{\omega(q)} \frac{x (\log_2 x)^{O(1)}}{q^M \log x} \largesum_{\substack{m \le x\\P_J(m) \le y}} \frac1m.$$
As a consequence, \eqref{eq:Sum1/mOverPJ(m)ley} yields
$$\Sigma_{r+1} \ll \frac{\lambda^{\omega(q)} x}{q^M (\log x)^{1/2}} \exp\left(O\left((\log_3 x)^2\right)\right) \ll \frac x{q^M (\log x)^{1/3}},$$
yielding the desired bound for all of $\Sigma_2, \dots , \Sigma_M$.

\textit{Bounding $\Sigma$:} Any $n$ counted in $\Sigma$ has $2M$ many prime factors (counted with multiplicity) exceeding $q$, out of which at most one of them can exactly divide $n$. Hence $\OmStqn \ge 2M-1$, and by the same argument as given above, any $n$ counted in $\Sigma$ can be expressed in the form $m p_1^{\beta_1} \cdots p_s^{\beta_s} P$, where $P \coloneqq P(n)>z$, $p_1, \dots , p_s>q$ are primes, $P_J(m) \le y$, and $\beta_1 \ge \alpha_1, \dots , \beta_s \ge \alpha_s$ for some positive integers $\alpha_1, \dots , \alpha_s$ at least $2$ summing to either $2M-1$ or $2M$. Given $m, s, \alpha_1, \dots , \alpha_s, p_1, \dots , p_s, \beta_1, \dots , \beta_s$, the number of possible $P$ is $\ll x/mp_1^{\beta_1} \cdots p_s^{\beta_s} \log z$. As above, we have $\sum_{i=1}^s \alpha_i - s \ge \max\{s, 2M-1-s\} \ge M$, so that the sum over $s, \alpha_1, \dots , \alpha_s, p_1, \dots , p_s, \beta_1, \dots , \beta_s$ is $O(q^{-M})$. Finally, using \eqref{eq:Sum1/mOverPJ(m)ley} to bound the sum on $m$, we obtain $\Sigma \ll x/q^M(\log x)^{1/3}$. 

This completes the proof of \eqref{eq:SqfreeqUnif_Remaining}, and hence that of Theorem \ref{thm:restrictediput_squarefreeq}. \hfill \qedsymbol
\subsection{Optimality in the input restrictions in Theorem \ref{thm:restrictediput_squarefreeq}:}\label{subsec:OptimalitySqfree_RestrictedInput}
For any $M \ge 2$, we construct additive functions $g_1, \dots , g_M$ showing that the restriction $P_{2M}(n)>q$ cannot be weakened to $P_{2M-3}(n)>q$ in our range of $q$. For $M=2$, the condition $P_{2M-3}(n)>q$ translates to $P(n)>q$; by known estimates on smooth numbers (\cite[Theorem 5.13\text{ and }Corollary 5.19, Chapter III.5]{tenenbaum15}), this latter condition may be ignored up to a negligible error, so the first counterexample in subsection \ref{subsec:OptimalityUnrestricted} suffices. 

Now assume that $M \ge 3$; consider additive functions $g_1, \dots , g_M: \NatNos \rightarrow \Z$ defined by the polynomials $G_i(T) \coloneqq (T-1)^i$, and satisfying the conditions $g_i(p^2) \coloneqq 0$ for all primes $p$ and all $i \in [M]$. As observed in subsection \ref{subsec:OptimalityUnrestricted}, the polynomials $\{G_i'\}_{i=1}^M$ are $\Q$-linearly independent, and with $\COG$ as in \eqref{eq:COGDefn}, we have $q \in \Qgfam$ for all moduli $q$ having $P^-(q)>\COG$. 

We see that $G_i(p) \equiv 0 \pmod q$ for all $i$ and for all primes $p \equiv 1 \pmod q$. Consequently, if $p_1, \dots , p_{M-2}, P$ are primes satisfying $q< p_{M-2} < \cdots < p_1 < x^{1/(4M-8)} <$ $x^{1/3} < P \le x/(p_1 \cdots p_{M-2})^2$ and $P \equiv 1 \pmod q$, then the positive integer $n \coloneqq (p_1 \cdots p_{M-2})^2 P$ is less than or equal to $x$, has $P_{2M-3}(n) > q$ and satisfies the conditions $g_i(n) = G_i(P) + \sum_{j=1}^{M-2} g_i(p_j^2) \equiv 0 \pmod q$ for all $i \in \{1, \dots , M\}$. By the Siegel--Walfisz Theorem, we find that
\begin{equation*}
\begin{split}
\largesum_{\substack{n \le x: ~ P_{2M-3}(n)>q\\(\forall i) ~ g_i(n) \equiv 0 \pmod q}} 1  ~ &\ge \largesum_{\substack{q < p_{M-2} < \cdots < p_1 < x^{1/(4M-8)}}} ~ \largesum_{\substack{x^{1/3} < P \le x/(p_1 \cdots p_{M-2})^2\\ P \equiv 1 \pmod q}} 1\\
&\gg \largesum_{\substack{q < p_{M-2} < \cdots < p_1 < x^{1/(4M-8)}}} \left( \frac x{\phi(q) (p_1 \cdots p_{M-2})^2 \log x} + O(x^{1/3})\right)\\
&\gg \frac x{q\log x} \largesum_{\substack{p_1, \dots , p_{M-2}\text{ distinct}\\ q< p_1, \dots , p_{M-2} < x^{1/(4M-8)}}} \frac1{(p_1 \cdots p_{M-2})^2}
\end{split}
\end{equation*}
Ignoring the distinctness condition in the sum above incurs a total error
$$\ll \frac x{q\log x} \largesum_{p_1, p_2, \dots , p_{M-3}>q} \frac 1{p_1^4 p_2^2 \cdots p_{M-3}^2} \ll \frac x{q\log x} \left(\largesum_{p>q} \frac 1{p^4}\right) \left(\largesum_{p>q} \frac 1{p^2}\right)^{M-4} \ll \frac x{q^M \log x}.$$
On the other hand, 
$$\largesum_{\substack{p_1, \dots , p_{M-2} \in (q, x^{1/(4M-8)})}} \frac1{(p_1 \cdots p_{M-2})^2} = \left(\largesum_{q<p<x^{1/(4M-8)}} \frac1{p^2}\right)^{M-2} \gg \frac1{(q \log q)^{M-2}}.$$
Collecting estimates, we obtain for all sufficiently large $q$,
$$\largesum_{\substack{n \le x: ~ P_{2M-3}(n)>q\\(\forall i) ~ g_i(n) \equiv 0 \pmod q}} 1 \gg \frac x{q^{M-1}\log x (\log q)^{M-2}} + O\left(\frac x{q^M \log x}\right) \gg \frac x{q^{M-1}\log x (\log_2 x)^{M-2}},$$
which grows strictly faster than $x/q^M$ as soon as $q> \log x \cdot (\log_2 x)^{M-1}$ (say). We conclude that the condition $P_{2M}(n)>q$ cannot be replaced by $P_{2M-3}(n)>q$ for \textit{any} $M \ge 2$.

One might wonder whether one of the conditions $P_{2M-1}(n)>q$ or $P_{2M-2}(n)>q$ could possibly suffice to restore uniformity in squarefree $q \le (\log x)^K$. In this direction, we now construct an example showing that the condition $P_{2M-2}(n)>q$ is also insufficient for $M=2$. Indeed, let consider additive functions $g_1, g_2$ defined by the polynomials $G_1(T) \coloneqq T$ and $G_2(T) \coloneqq T^{3}$, so that $\{G_1', G_2'\}$ are clearly $\Q$-linearly independent. With $\COG$ as usual, we have $q \in \mathcal Q_{(g_1, g_2)}$ for all $q$ having $P^-(q)>\COG$. 

However, if $n$ is of the form $P_1 P_2$ for distinct primes $P_1, P_2>y \coloneqq \exp((\log x)^{1/2})$ satisfying $P_2 \equiv - P_1 \pmod q$, then $P_2(n)>y>q$, while $G_i(P_1) + G_i(P_2) \equiv 0 \pmod q$ for $i \in \{1, 2\}$, so that $g_1(n) \equiv g_2(n) \equiv 0 \pmod q$. As such, for $2<q \le (\log x)^K$, a simpler version of the arguments leading to \eqref{eq:CongrueceRemoved} yields
\begin{equation}\label{eq:P2M-2FirstEstim}
\begin{split}
\largesum_{\substack{n \le x:~ P_2(n)>q\\(\forall i) ~ g_i(n) \equiv 0 \pmod q}} 1 ~ ~ &\ge \largesum_{v \in U_q} ~ \frac1{2!}\largesum_{\substack{P_1, P_2>y\\P_1 \ne P_2, ~ P_1 P_2 \le x\\P_1 \equiv v, ~ P_2 \equiv -v \pmod q}} 1\\
&\gg \frac1{\phi(q)}\largesum_{\substack{P_1, P_2>y: ~P_1 P_2 \le x}} 1 + O(x\exp(-C'(\log x)^{1/4})) \gg \frac{x \log_2 x}{q \log x},
\end{split}
\end{equation}
where $C' \coloneqq C'(K)>0$ is a constant, and the last bound above is a simple consequence of Chebyshev's and Mertens' estimates. In particular, this shows that the tuple $(0, 0)$ mod $q$ is overrepresented by $(g_1, g_2)$ once $q>\log x/(\log_2 x)^{1/2}$, showing failure of uniformity in squarefree $q$ after a very small threshold, under the restriction $P_{2M-2}(n)>q$ for $M=2$. 

It is to be noted that our arguments above go through for any two polynomials $G_i(T) \coloneqq A_i T^{k_i} + B_i$ ($i \in \{1, 2\}$), for any two \textit{distinct} \textit{odd} positive integers $k_i$, and any integers $A_i \ne 0$ and $B_i$. Indeed, the distinctness of $k_1$ and $k_2$ ensures that $G_1'$ and $G_2'$ are $\Q$-linearly independent, while their parity ensures that any two primes $P_1, P_2$ satisfying $P_2 \equiv -P_1 \pmod q$ also satisfy $G_i(P_1) + G_i(P_2) \equiv 2B_i \pmod q$ for both $i \in \{1, 2\}$. As such, the above arguments show that there are $\gg x\log_2 x/q\log x$ many $n \le x$ satisfying $g_i(n) \equiv 2B_i \pmod q$ for $i \in \{1, 2\}$. This gives an infinite family of counterexamples showing that the condition $P_{2M-2}(n)>q$ is not sufficient to restore uniformity in squarefree $q \le (\log x)^K$ in the case $M=2$.

In conclusion, this means that our restriction $P_{2M}(n)>q$ in Theorem \ref{thm:restrictediput_squarefreeq} is at most ``one step away" from optimal, in the sense that it might still be possible to weaken it to $P_{2M-1}(n)>q$.
\section{Necessity of the linear independence hypothesis: Proof of Theorem \ref{thm:LIHNecess}}
Recall that the $\Q$-linear independence of $\{G_i'\}_{i=1}^{M-1}$ is equivalent to that of $\{G_i-G_i(0)\}_{i=1}^{M-1}$; likewise, the condition $G_M' = \sum_{i=1}^{M-1} a_i G_i'$ is exactly equivalent to the condition $G_M(T) - G_M(0) = \sum_{i=1}^{M-1} a_i (G_i(T)-G_i(0))$ in the ring $\Q[T]$. We claim that the polynomials $\{G_i\}_{i=1}^M$ are $\Q$-linearly independent. Indeed, suppose there exist integers $\beta_1, \dots , \beta_M$ for which $\sum_{i=1}^M \beta_i G_i(T) = 0$ in $\Q[T]$. Since $G_M(T) = G_M(0) + \sum_{i=1}^{M-1} a_i (G_i(T)-G_i(0))$, we find that
\begin{equation}\label{eq:LDrel}
\sum_{i=1}^{M-1} (\beta_i + \beta_M a_i) G_i(T) = \beta_M  \left(\sum_{i=1}^{M-1} a_i G_i(0) - G_M(0) \right),    
\end{equation}
so that $\sum_{i=1}^{M-1} (\beta_i + \beta_M a_i) (G_i(T) - G_i(0)) = 0$. Since $\{G_i(T)-G_i(0)\}_{i=1}^{M-1}$ are $\Q$-linearly independent, the last relation forces $\beta_i = -\beta_M a_i$ for all $i \in \{1, \dots , M-1\}$, which by \eqref{eq:LDrel} leads to
$$\beta_M \left(\sum_{i=1}^{M-1} a_i G_i(0) - G_M(0) \right) = 0.$$
Now if $\beta_M \ne 0$, then the above relation forces $\sum_{i=1}^{M-1} a_i G_i(0) = G_M(0)$ contrary to hypothesis. Hence, we must have $\beta_M = 0$, forcing $\beta_i = -\beta_M a_i = 0$ for all $i \in \{1, \dots , M-1\}$. This shows that $\{G_i\}_{i=1}^M$ are indeed $\Q$ linearly independent. 

As such by Corollary \ref{cor:JtEquidLargePD}(i) and the discussion preceding it, there exists a constant $C_1(\widehat{G})>0$ such that $\{G_i\}_{i=1}^M$ are $\F_\ell$-linearly independent for all $\ell>C_1(\widehat{G})$, and so $Q \in \Qgfam$ for all moduli $Q>1$ having $P^-(Q)>C_1(\widehat{G})$. In addition, since $\{G_i'\}_{i=1}^{M-1}$ are $\Q$-linearly independent, there exists (by \eqref{eq:COGDefn}) a constant $C_0(G_1, \dots , G_{M-1})>0$ such that $\{G_i'\}_{i=1}^{M-1}$ are $\F_\ell$-linearly independent for any $\ell>C_0(G_1, \dots , G_{M-1})$.

We set $\CGsub$ to be any constant exceeding $\max\{C_1(\widehat{G}), 4M(32D)^{2D+4}, C_0(G_1, \dots , G_{M-1})\}$ and henceforth consider moduli $q$ having $P^-(q)>\CGsub$, so that $q \in \Qgfam$. Given any $R>\CGsub$ and integers $\{b_i\}_{i=1}^{M-1}$, set $b_M \coloneqq G_M(0) R + \sum_{i=1}^{M-1} a_i (b_i - G_i(0)R)$. Then the relations $\sum_{j=1}^R G_i(v_j) \equiv b_i \pmod q$ for $i \in \{1, \dots , M-1\}$ also imply that $\sum_{j=1}^R G_M(v_j) \equiv b_M \pmod q$. As such, for any $R$ distinct primes $P_1, \dots , P_R$, with $(P_1, \dots , P_R)$ mod $q$ lying in the set 
$$V \coloneqq \mathcal V_{R, M-1}\left(q; (b_i)_{i=1}^{M-1}\right) = \left\{(v_j)_{j=1}^R \in (U_q)^R: (\forall i \in [M-1]) ~ \sum_{j=1}^R G_i(v_j) \equiv b_i \pmod q \right\},$$
we have $g_i(P_1 \cdots P_R) \equiv b_i \pmod q$ for \textit{all} $i \in [M]$. Letting $y \coloneqq \exp((\log x)^{1/2})$, a simpler version of the arguments leading to \eqref{eq:CongrueceRemoved} yields, for $q \le (\log x)^K$, 
\begin{align*}
\largesum_{\substack{n \le x:~ P_R(n)>q\\(\forall i) ~ g_i(n) \equiv b_i \pmod q}} 1 ~ ~ &\ge \largesum_{(v_1, \dots , v_R) \in V} ~ \frac1{R!}\largesum_{\substack{P_1, \dots , P_R>y\\P_1 \cdots P_R \le x\\P_1, \dots , P_R\text{ distinct}\\(\forall j) ~ P_j \equiv v_j \pmod q}} 1\\
&\gg \frac{\#V}{\phi(q)^R}\largesum_{\substack{P_1, \dots , P_R>y\\P_1 \cdots P_R \le x\\P_1, \dots , P_R\text{ distinct}}} 1 + O(x\exp(-C'(\log x)^{\delta/4}))\\ &\gg \frac{\#V}{\phi(q)^R}\largesum_{\substack{P_1, \dots , P_R>y\\P_1 \cdots P_R \le x}} 1 + O(x\exp(-C'(\log x)^{\delta/4}))
\end{align*}
for some constant $C' \coloneqq C'(K)>0$. A direct induction on $R$ (involving Chebyshev's estimate) shows that the last sum above is 
$$\largesum_{\substack{n \le x: ~ P^-(n)>y\\\Omega(n)=R}} 1 ~ \gg ~ \frac{x (\log_2 x)^{R-1}}{\log x},$$
leading to 
\begin{equation*}
\largesum_{\substack{n \le x:~ P_R(n)>q\\(\forall i) ~ g_i(n) \equiv b_i \pmod q}} 1 ~ ~ \gg ~ \frac{\#V}{\phi(q)^R} \cdot \frac{x (\log_2 x)^{R-1}}{\log x} + O(x\exp(-C'(\log x)^{\delta/4})).
\end{equation*}
As such, 
to complete the proof of the theorem, it remains to show that 
\begin{equation}\label{eq:VRqb1}
\#V = \#\mathcal V_{R, M-1}\left(q; (b_i)_{i=1}^{M-1}\right) \gg \frac{\phi(q)^R}{q^{M-1}}. 
\end{equation}
To show this, we argue as in the proof of the estimate \eqref{eq:Largeprime}: for each prime power $\ell^e \parallel q$, we write
\begin{multline*}
\#\mathcal V_{R, M-1}\left(\ell^e; (b_i)_{i=1}^{M-1}\right)\\ = \frac{\phi(\ell^e)^R}{\ell^{e(M-1)}} \left\{1+\frac1{\phi(\ell^e)^R} \largesum_{(r_1, \dots , r_{M-1}) \not\equiv (0, \dots , 0) \bmod \ell^e} e\left(-\frac1{\ell^e}\largesum_{i=1}^{M-1} r_i b_i\right) (Z_{\ell^e; \, r_1, \dots , r_{M-1}})^R\right\},
\end{multline*}
where $Z_{\ell^e; \, r_1, \dots , r_{M-1}} \coloneqq \largesum_{v \bmod \ell^e} \chi_{0, \ell}(v) e\left(\frac1{\ell^e}\largesum_{i=1}^{M-1} r_i G_i(v)\right)$ for each $(r_1, \dots , r_{M-1}) \not\equiv (0, \dots , 0) \bmod \ell^e$. Since $(r_1, \dots , r_{M-1}) \not\equiv (0, \dots , 0) \bmod \ell^e$, we have $\gcd(\ell^e, r_1, \dots , r_{M-1}) = \ell^{e-e_0}$ for some $1 \le e_0 \le e$ and $|Z_{\ell^e; \, r_1, \dots , r_{M-1}}| \le D\ell^{e-e_0/D}$ (here it is important that since $\ell>C_{\widehat{G}}$, the polynomials $\{G_i'\}_{i=1}^{M-1}$ are $\F_\ell$-linearly independent). We obtain
\begin{align*}
\frac1{\phi(\ell^e)^R} \largesum_{(r_1, \dots , r_{M-1}) \not\equiv (0, \dots , 0) \bmod \ell^e} |Z_{\ell^e; \, r_1, \dots , r_{M-1}}|^R &\le \frac{D^R \ell^{eR}}{\phi(\ell^e)^R} \largesum_{e_0 \ge 1} \left(\ell^{M-1-R/D}\right)^{e_0} \le \frac{2 (2D)^R}{\ell^{R/D-M+1}}.
\end{align*}
Since $R/D-M \ge R/(D+2)$ and $\ell^{1/(2D+4)} > (\CGsub)^{1/(2D+4)} > 32D$, this leads to
\begin{equation*}
\begin{split}
\frac1{\phi(\ell^e)^R} \largesum_{(r_1, \dots , r_{M-1}) \not\equiv (0, \dots , 0) \bmod \ell^e} |Z_{\ell^e; \, r_1, \dots , r_{M-1}}|^R &\le \frac{2(2D)^R}{\ell^{R/(D+2)}}\\ &\le \frac{2(2D)^R}{(32D)^R} \cdot \frac1{\ell^{R/(2D+4)}} \le \frac1{8^R \ell^{R/(2D+4)}} \le \frac1{8\ell^2}.
\end{split}
\end{equation*} 
Hence, for each prime power $\ell^e \parallel q$, 
\begin{equation}\label{eq:VR1PrimePower}
\#\mathcal V_{R, M-1}\left(\ell^e; (b_i)_{i=1}^{M-1}\right) \ge \frac{\phi(\ell^e)^R}{\ell^{e(M-1)}} \left(1-\frac1{8 \ell^2}\right),
\end{equation}
and since $\prod_{\ell \mid q} \left(1-\frac1{8 \ell^2}\right) \ge 1-\frac18 \sum_{\ell \ge 2}\frac1{\ell^2} \ge \frac78$, we obtain by multiplying all the bounds \eqref{eq:VR1PrimePower},
$$\#V = \prod_{\ell^e \parallel q} \#\mathcal V_{R, M-1}\left(\ell^e; (b_i)_{i=1}^{M-1}\right) \ge \frac78 \cdot \frac{\phi(q)^R}{q^{M-1}}.$$
This shows \eqref{eq:VRqb1}, completing the proof of Theorem \ref{thm:LIHNecess}, and demonstrating the necessity of the linear independence hypothesis in the generality of our setting. \hfill \qedsymbol

\section*{Acknowledgements}
This work was done in partial fulfillment of my PhD at the University of Georgia. As such, I would like to thank my advisor, Prof. Paul Pollack, for the past joint research that has led me to think about this question, as well as for his continued support and encouragement. I would also like to thank the Department of Mathematics at UGA for the Teaching Assistantship awarded by them, as well as for their support and hospitality. 

\providecommand{\bysame}{\leavevmode\hbox to3em{\hrulefill}\thinspace}
\providecommand{\MR}{\relax\ifhmode\unskip\space\fi MR }
\providecommand{\MRhref}[2]{
  \href{http://www.ams.org/mathscinet-getitem?mr=#1}{#2}
}
\providecommand{\href}[2]{#2}


\begin{thebibliography}{10}

\bibitem{akande23} A. ~Akande, \emph{Uniform distribution of polynomially-defined additive functions to varying moduli}, submitted. 

\bibitem{AE77} 
K. ~Alladi and P. Erd\H{o}s, \emph{On an additive arithmetic function}, Pacific J. Math. 71 (1977), no. 2, 275--294. 

\bibitem{cochrane02}
T.~Cochrane, \emph{Exponential sums modulo prime powers}, Acta Arith.
  101 (2002), 131--149.

\bibitem{CZ99} T. ~Cochrane and Z. Zheng, \emph{Pure and mixed exponential sums.}, Acta Arith. 91 (1999), 249--278. 

\bibitem{delange69}
H.~Delange, \emph{On integral-valued additive functions}, J. Number Theory
  1 (1969), 419--430.

\bibitem{delange74}
H.~Delange, \emph{On integral-valued additive functions, II}, J. Number Theory
6 (1974), 161--170.

\bibitem{goldfeld17}
D.~Goldfeld, \emph{On an additive prime divisor function of {A}lladi and
  {{E}rd\H{o}s}}, Analytic number theory, modular forms and
  {$q$}-hypergeometric series, Springer Proc. Math. Stat., vol. 221, Springer,
  Cham, 2017, pp.~297--309.

\bibitem{HT88}
R.R. Hall and G.~Tenenbaum, \emph{Divisors}, Cambridge Tracts in Math., vol.~90, Cambridge Univ. Press, Cambridge, 1988.

\bibitem{HW08}
G.H.~Hardy and E.M.~Wright, \emph{An introduction to the theory of numbers}, sixth ed., Oxford Univ. Press, Oxford, 2008.

\bibitem{konyagin79b}
S.~Konyagin, \emph{Letter to the editors: ``{T}he number of solutions of
  congruences of the {$n$}th degree with one unknown''}, Mat. Sb. (N.S.)
  110(152) (1979), 158.

\bibitem{konyagin79a}
\bysame, \emph{The number of solutions of congruences of the {$n$}th degree
  with one unknown}, Mat. Sb. (N.S.) 109(151) (1979), 171--187, 327.

\bibitem{LPR21}
N.~Lebowitz-Lockard, P.~Pollack, and A.~Singha~Roy, \emph{Distribution mod $p$
  of {E}uler's totient and the sum of proper divisors}, Michigan Math. J., to appear. 

\bibitem{LY94}
D.B. Leep and C.C. Yeomans, \emph{The number of points on a singular curve over
  a finite field}, Arch. Math. (Basel) 63 (1994), 420--426.

\bibitem{MV07}
H.L. Montgomery and R.C. Vaughan, \emph{Multiplicative number theory. {I}.
  {C}lassical theory}, Cambridge Studies in Advanced Mathematics, vol.~97,
  Cambridge Univ. Press, Cambridge, 2007.

\bibitem{pillai40}
S.S. Pillai, \emph{Generalisation of a theorem of {Mangoldt}}, Proc. Indian
  Acad. Sci., Sect. A 11 (1940), 13--20.

\bibitem{PSR}
P.~Pollack and A.~Singha~Roy, \emph{Benford behavior and distribution in
  residue classes of large prime factors}, Canad. Math. Bull. 66 (2023), no. 2, 626--642. MR 4584488

\bibitem{PSR22}
\bysame, \emph{Joint distribution in residue classes of polynomial-like
 multiplicative functions}, Acta Arith. 202 (2022), 89--104.

\bibitem{PSR23}
\bysame, \emph{Distribution in coprime residue classes of polynomially-defined multiplicative functions}, Math. Z. 303 (2023), no. 4, Paper No. 93, 20 pages.

\bibitem{schmidt76}
W.M. ~Schmidt, \emph{Equations over finite fields}, L.N.M. 536, Springer-Verlag, Berlin, (1976).

\bibitem{tenenbaum15}
G.~Tenenbaum, \emph{Introduction to analytic and probabilistic number theory},
  third ed., Graduate Studies in Mathematics, vol. 163, American Mathematical
  Society, Providence, RI, 2015.

\bibitem{weil48}
A. ~Weil, \emph{On some exponential sums}, Proc. Nat. Acad. Sci. U.S.A. 34 (1948), 203-210.

\end{thebibliography}
\end{document}